\documentclass[12pt]{amsart}
\usepackage{etex} \usepackage{amssymb,amsthm,amsmath,amsfonts,stmaryrd,soul}

\usepackage[bb=pazo]{mathalpha}

\usepackage[colorlinks,anchorcolor=black,citecolor=black,linkcolor=black,filecolor=black,urlcolor=black]{hyperref}

\usepackage[iso,hhmmss]{datetime}

\usepackage{xargs}                      \usepackage{xcolor}  \usepackage[colorinlistoftodos,textsize=tiny]{todonotes}
\usepackage{todonotes}
\newcommandx{\question}[2][1=]{\todo[linecolor=blue,backgroundcolor=blue!25,bordercolor=blue,#1]{#2}}
\newcommandx{\mihai}[2][1=]{\todo[linecolor=teal,backgroundcolor=teal!25,bordercolor=teal,#1]{#2}}
\newcommandx{\franz}[2][1=]{\todo[linecolor=teal,backgroundcolor=green!25,bordercolor=teal,#1]{#2}}
\newcommandx{\kamil}[2][1=]{\todo[linecolor=teal,backgroundcolor=purple!25,bordercolor=teal,#1]{#2}}
\newcommandx{\wiktor}[2][1=]{\todo[linecolor=teal,backgroundcolor=magenta!25,bordercolor=teal,#1]{#2}}
\newcommandx{\info}[2][1=]{\todo[caption={Short note},linecolor=violet,backgroundcolor=violet!25,bordercolor=violet,size=\tiny,#1]{#2}}
\newcommandx{\thiswillnotshow}[2][1=]{\todo[disable,#1]{#2}}

\newcommand{\seg}{\mathrm{seg}}

\usepackage{color}

\usepackage{pict2e}

\usepackage[margin=2cm,a4paper]{geometry}

\usepackage{tikz,pgfplots,pgfplotstable}

\usetikzlibrary{patterns}
\usetikzlibrary{backgrounds}

\usepackage{ifthen}
\newboolean{coloron}
\setboolean{coloron}{false}
\ifthenelse{\boolean{coloron}}
{
	\tikzset{    afill/.style={draw,color=red,fill=red},
		bfill/.style={draw,color=green,fill=green}
	}
	\tikzstyle{thline}=[thick]
	\tikzstyle{cline}=[color=red]
	\tikzstyle{aline}=[color=red]
	\tikzstyle{bline}=[color=green]
}
{
	\tikzset{      afill/.style={draw,pattern=crosshatch},
		bfill/.style={draw,pattern=crosshatch dots}
	}
	\tikzstyle{thline}=[thick]
	\tikzstyle{cline}=[densely dashdotted]
	\tikzstyle{aline}=[densely dashdotted]
	\tikzstyle{bline}=[densely dotted]
}

\usepackage{mathabx}
\newlength{\tmpl}

\newcommand{\Prodab}{\alpha}

\usepackage[extra]{tipa}
\usepackage{graphicx}
\usepackage{caption,subcaption}

\usepackage[shortlabels]{enumitem}

\numberwithin{equation}{section}

\newcommand{\uComp}[1]{\mathcal{C}^\vee_{#1}} 
\newcommand{\oneconn}[1]{\mathcal{C}^{1}_{#1}}
\newcommand{\oneconntype}[2]{\mathcal{C}^{1,#2}_{#1}}
\newcommand{\twoconn}[1]{\mathcal{C}^{2}_{#1}}

\newcommand{\CompOdd}[1]{\twoconn{#1}}

\DeclareMathOperator{\Int}{\mathit{Int}}

\newcommand{\harpleftsign}{\scriptstyle\leftharpoonup}
\newcommand{\harpleft}[2]{  \ifx\displaystyle#1\doalign{$\harpleftsign$}{#1#2}\fi
	\ifx\textstyle#1\doalign{$\harpleftsign$}{#1#2}\fi
	\ifx\scriptstyle#1\doalign{\scalebox{.6}[.9]{$\harpleftsign$}}{#1#2}\fi
	\ifx\scriptscriptstyle#1\doalign{\scalebox{.5}[.8]{$\harpleftsign$}}{#1#2}\fi
}

 \newcommand{\harprightsign}{\scriptstyle\rightharpoonup}
\newcommand{\harpright}[2]{  \ifx\displaystyle#1\doalign{$\harprightsign$}{#1#2}\fi
	\ifx\textstyle#1\doalign{$\harprightsign$}{#1#2}\fi
	\ifx\scriptstyle#1\doalign{\scalebox{.6}[.9]{$\harprightsign$}}{#1#2}\fi
	\ifx\scriptscriptstyle#1\doalign{\scalebox{.5}[.8]{$\harprightsign$}}{#1#2}\fi
}

\newcommand{\doalign}[2]{ {\vbox{\offinterlineskip\ialign{\hfil##\hfil\cr#1\cr$#2$\cr}}}}

\def\R{{\mathbb R}}

\def\IC{{\mathbb C}}

\def\IN{{\mathbb N}}

\def\<{\langle}
\def\>{\rangle}

\def\X{{\mathnormal X}}

\def\A{\mathcal{A}}

\def\SpPair{\mathrm{piv}}

 \DeclareMathOperator{\SG}{\mathfrak{S}} \DeclareMathOperator{\NC}{\mathit{NC}}

\newtheorem{th-def}{Theorem-Definition}[section]
\newtheorem{theo}{Theorem}[section]
\newtheorem{lemm}[theo]{Lemma}

\newtheorem{prop}[theo]{Proposition}
\newtheorem{cor}[theo]{Corollary}
\theoremstyle{definition}

\newtheorem{defi}[theo]{Definition}

\newtheorem{Rem}[theo]{Remark}

\def\cvput#1[#2]{\pnode(#1,1){#1} \pscircle*(#1,1){.1} \rput(#1,.5){$#2$}}

\title[Quadratic polynomials preserving free infinite divisibility]{Quadratic polynomials preserving free infinite divisibility}
	\author[W. Ejsmont]{Wiktor Ejsmont}
	\address[Wiktor Ejsmont]{ Department of Telecommunications and Teleinformatics, 
		Wrocław University of Science and Technology \\ Wybrze\.ze Wyspia{\'n}skiego 27, 50-370 Wroc\l aw, Poland }
	\email{wiktor.ejsmont@gmail.com}
\author[F. Lehner]{Franz Lehner}
\address[Franz Lehner]{Institut f\"ur Diskrete Mathematik,
	Technische Universit\"at Graz,
	Steyrergasse 30, 8010 Graz, Austria }
\email{lehner@math.tugraz.at}

\author[M. Popa]{Mihai Popa}
\address[Mihai Popa]{The University of Texas at San Antonio, Department
	of Mathematics, One UTSA Circle, San Antonio, Texas 78249,
	and \newline ${}\hspace{.35cm}{}$ Institute of Mathematics
	``Simion Stoilow'' of the Romanian Academy, P.O. Box
	1-764, Bucharest, RO-70700, Romania}

\author[K. Szpojankowski]{Kamil Szpojankowski}
\address[Kamil Szpojankowski]{Institute of Mathematics of the Polish Academy of Sciences, ul. \'Sniadeckich 8, 00-656 Warszawa, Poland, and 
	Faculty of Mathematics and Information Science,
	Warsaw University of Technology, Koszykowa 75, 00-662 Warszawa, Poland.} 
\email{kszpojankowski@impan.pl, kamil.szpojankowski@pw.edu.pl  }
\thanks{KSz: This research was funded in part by National Science Centre, Poland WEAVE-UNISONO grant BOOMER 2022/04/Y/ST1/00008.}

\begin{document}

\begin{abstract} 
  We characterize  non-commutative polynomials of degree 2 which preserve free
  infinite divisibility as self-adjoint Lie polynomials of degree 2.
  This extends previous results which characterized homogeneous
  quadratic forms only.
  The proof is based on an involution   of colored noncrossing partitions
  cancelling many terms;
  positive definiteness of the remaining terms is established
  by an explicit operator model on Fock space.
  Examples show that higher order Lie polynomials do not preserve free infinite
  divisibility.

\end{abstract}

\date{
\today, \currenttime}

\maketitle

\section{Introduction}

We study quadratic polynomials in non-commutative variables,
that is polynomials of the form
\begin{align}\label{eqn:1.1}
  P(X_1,X_2,\ldots,X_n)=\sum_{i,j=1}^na_{i,j}X_i X_j+\sum_{i =1}^n b_iX_i
  .
\end{align} 
A self-adjoint non-commutative polynomial is said to \emph{preserve free infinite
  divisibility}
if it produces a freely infinitely
divisible random variable whenever
it is evaluated in mutually free and freely infinitely divisible random
variables.

In \cite{EjsmontLehner:2019:commutators}
homogeneous polynomials of degree 2 with this property were
characterized as sums of commutators or, equivalently,
quadratic forms with  self-adjoint skew-symmetric
system matrix $A=[a_{i,j}]_{i,j=1}^N$.
In the present paper we show that this characterization remains true
if an arbitrary linear part is added. 
After a presentation of the results from \cite{EjsmontLehner:2019:commutators}
D.-V. Voiculescu suggested the question whether arbitrary self-adjoint Lie
polynomials preserve infinite divisibility.
While the answer to this question is negative in general
(see the examples of degree 3 in Remark~\ref{computer},
first of which was given by D. Shlyakhtenko),
the results of the present paper show that it is true for Lie polynomials of
degree two.

For the proof we extend the sign-reversing involution on  colored noncrossing
partitions  from 
\cite{EjsmontLehner:2019:commutators} to the nonhomogeneous setting.
This gives a concise description of the free cumulants,
and a full Fock-space realization identifies the surviving
terms with moments of a selfadjoint operator. Our operator model is inspired
by, but not similar to the ones from \cite{popavinnikov} and
\cite{popapacific}, and more streamlined and general than the one from
\cite{PopaSzpojankowski:2025}. This combination of combinatorial cancellation
and operator construction provides a structural explanation for the
preservation phenomenon and may be useful in studying higher-degree
transformations of freely infinitely divisible laws. 

The paper is organized as follows.

In Section 2 we provide the necessary background
and give some new notations and definitions regarding noncrossing partitions
and free cumulants.

Section 3  presents the main combinatorial result, namely
the vanishing of the contribution of a class of noncrossing partitions in the
expansion of the free cumulants of $P(X_1,X_2,\ldots,X_N)$. The results in this
section are proved under the hypothesis that the matrix $A$ from (\ref{eqn:1.1}) is
skew symmetric, but notably we do not assume freeness of the variables $ X_1,
X_2, \dots, X_N$.

Section 4 is divided into 2 subsections. In the first subsection we contruct an
explicit operator $ V $  on a full Fock space without the vacuum vector such
that the $m$-th free cumulants of $ P$ equals the $m$-th moment $\langle V^m \xi, \xi \rangle$ 
with respect to a certain vector $ \xi$ and under the assumption that $
X_1,X_2, \dots, X_N $ are mutually free compound free Poisson variables.
The second subsection contains the main result of the paper, Theorem~\ref{prop:znikaniekumulant}, which characterizes the polynomials from (\ref{eqn:1.1}) which preserve free infinite divisibility. The last part of the paper, Remark~\ref{computer} gives some counterexamples showing that preservation of free infinite divisibility does not generally extend to the Lie algebra generated by $ X_1, X_2,\dots, X_N$.

\section{Preliminaries}

\subsection{Noncrossing Partitions, Free Cumulants and Free Infinite Divisibility}\label{prelim:1}

We shall identify finite ordered sets with $n $ elements with the ordered set $[ n ] = \{ 1, 2, \dots, n \}$.  A partition $\pi$ of $[n]$ is a set $\{ B_1, B_2, \dots, B_k\}$ of mutually disjoint subsets of $ [n]$
(also called \emph{blocks}) 
whose union is $[n]$. The \emph{size} of $\pi$ is  the number of blocks of
$\pi$, denoted by $\lvert\pi\rvert=k$.
A block consisting of only one elements is called a \emph{singleton}. 

 A set of consecutive elements $I =\{i, i+1, \dots, i+k\}$
of $ [n] $ is said to be a \emph{segment} of a partition $\pi $ of $[n]$ if
there exists a block $B$ of $\pi$ such that $I \subseteq B $, but $ i-1$ and
$i+k+1$  are not elements of $ B$. In other words segments are the maximal
intervals contained in the blocks of $\pi$.
First and last segments of $\pi$, in lexicographic order, will be called  the \emph{extremal segments} of $\pi$.

The interval partition consisting of these segments is called the
\emph{segmentation} of $\pi$.

	The action of the symmetric group $\SG_n$ on the set $[n]$ naturally
induces an action on set partitions, namely
\begin{equation}
	\label{eq:sigma.pi}
	\sigma\cdot\pi=\{\sigma(B) \mid B\in\pi\}
	.
\end{equation}

 A partition $\pi$ is called \emph{noncrossing} 
if different blocks do not interlace, i.e., there is no quadruple of
elements $i<j<k<l$ such that $i$ is the same block as $k$, and $j$ is in the same block as $l$, but $ i $ and $j $ are not in the same block. 
The set of noncrossing partitions of $[n]$ is denoted by $\NC(n)$. 
$\NC(n)$ is a lattice under \emph{reverse refinement order}, where the relation $\pi\leq \rho$
holds if
every block of $\pi$ is contained in a block of $\rho$.
The maximal element of the lattice, i.e. the partition
 consisting of only one block, is denoted by $\hat{1}_{n}$; the minimal element, i.e. the partition consisting of $ n $ singletons, is denoted by $\hat{0}_n$.
 For more details and proofs on this topic see \cite[Lecture~9]{NicaSpeicher:2006}.

 \label{ssec:freecumulants}
 Given  a  non-commutative  $ \ast$-probability space $(\A,\tau)$  (i.e. a unital $\ast$-algebra
 $ \A $ endowed with a state $\tau $),
 the free  cumulants $\big( \kappa_n \big)_n $ are $ \mathbb{C} $-valued multilinear maps  on $ \A $
 defined  inductively via the relation
 \begin{align*}
   \tau(\X_{1}\X_{2}\dots \X_{n}) = \sum_{\pi \in \NC(n)}\kappa_{\pi}(\X_{1},\X_{2},\dots ,\X_{n}),\label{eq:DefinicjaKumulant}
 \end{align*}
 where 
 $\displaystyle
 \kappa_{\pi}(\X_{1},\X_{2},\dots ,\X_{n})
 =\prod_{\substack{B \in \pi\\ B = \{i_1, i_2, \dots i_q \}}}\kappa_{q}(\X_{i_1}, X_{i_2}, \dots, X_{i_q}). 
 $
 
 Univariate free cumulants $ \kappa_n(\X,\X, \dots. \X) $ will be abbreviated by $\kappa_n (X) $.
 
 Since the  non-commutative variables of concern in this paper are self-adjoint, we will list relevant properties of free cumulants in this framework.
A family $\{\X_1, \X_2, \dots, \X_n \} $  of self-adjoint elements from $(\A,\tau)$  is  \emph{free} (in the sense of Voiculescu see \cite{Voiculescu92}) if 
 \[ \kappa_m(\X_{h(1)}, \X_{h(2)}, \dots, \X_{h(m)} ) = 0 \]
 for every non-constant map $ h:[m]\rightarrow [n] $.

 Throughout the paper we will make extensive use of 
 the free version,
 due to Krawczyk and Speicher~\cite{KrawczykSpeicher:2000} (see also   \cite[Theorem 11.12]{NicaSpeicher:2006}), of the classical formula of 
 James and Leonov/Shiryaev \cite{James:1958,LeonovShiryaev:1959} which
 expresses cumulants of products in terms of  cumulants of the individual factors.
 \begin{theo}
   \label{thm:krawczyk}
   Assume that $\X_1, \dots, \X_n $ are (not necessarily distinct) non-commutative random variables and that $ 1\leq i_1 < i_2 < \dots < i_r =n$. Then
   \begin{align} 
     \label{twr:produktargumentow}
     \kappa_r(\X_1\dotsm \X_{i_1},X_{{i_1}+1}\dotsm X_{i_2},\dots,\X_{i_{r-1}+1}\dotsm\X_n)=\sum_{\substack{\pi\in \NC(n) \\ \pi\vee\rho=\hat{1}_{n}} }\kappa_\pi({\X_1,\dots,\X_n}). 
   \end{align} 
   where $\rho$ is the partition with blocks $\{1, 2, \dots, i_1\}$, $\{i_1+1, \dots, i_2\}, \dots, 
   \{i_{r-1}+1, \dots, i_r\}$.
 \end{theo}

 We will call the partition $\rho$ above, the \emph{product degree} partition.
	
  A self-adjoint element $ \X $ from the $\ast$-non-commutative
  probability space is said to have \emph{free infinitely divisible
    distribution} or to be \emph{free infinitely divisible} if for any
  positive integer $ n $ there exist some self-adjoint $ X_n $ such that
  the distribution of $\X $ is the same as the distribution of $n $ free
  copies of $ \X_n $, i.e. $ \kappa_m (\X) = n \kappa_m (X_n)$ for each
  positive integer $ m $.
  The free infinite divisibility of $ \X $ is
  equivalent to  the sequence $\big( \kappa_n(X)_n\big) $
  being conditionally positive definite
 (see, for
 example \cite{AnGao2019}, \cite{ArizmendiHasebeSakuma13},  or
 \cite[Chapter 13]{NicaSpeicher:2006})
  or to $\X $ being a limit (in distribution) of compound free Poisson random
  variables,
  i.e. of non-commutative random variables $ Y $ such that  there exist a compactly supported positive measure $ \sigma $ on $ \mathbb{R} $ satisfying 
	$ \displaystyle \kappa_m(Y) = \int_{\mathbb{R} } t^m d \sigma(t) $  for all $ m \geq 1 $.

\subsection{Interval and Fibonacci partitions}\label{notations:2}

An \emph{interval} block of a partition $\pi$
is a block of the form $V=\{k, k+1,\dots, k+l\}$ for $k\geq 1$ and $0 \leq l \leq n -k$.
A partition $\pi$ is called \emph{interval partition} if every block is an interval.
The set of interval partitions of $[n]$ is denoted by 
$\Int(n)$. 

The endpoints of  $ i_1 < i_2 < \dots < i_r = n$ of 
the intervals of an interval partition 
$$
\rho=\{\{1,2,\dots,i_1\},\{i_1+1,i_1+2,\dots,i_2\},\dots,\{i_{r-1}+1,i_{r-1}+2,\dots,i_r\}\}
\in\Int(n)
$$ 
are called the \emph{descents} of $\rho$.

For every $n\in\IN$ we denote by        $\Int_{1,2}(n)$
the set of interval partitions with blocks of size either 1 or 2.
It is an easy exercise that these  sets are enumerated by the
Fibonacci numbers and therefore we call its elements
\emph{Fibonacci partitions}\footnote{
  There is an immediate bijection between $\Int_{1,2}(n)$
  and the metres  in Pingala's description of Sanskrit prosody,
  which marks the first occurrence
  of Fibonacci numbers in history \cite{WilsonWatkins:2013}.
}. Interval blocks with 2 elements  are called
\emph{braces}.

The \emph{interval closure} $\hat{\pi}$ of a partition $\pi\in \NC(n) $ 
is the smallest interval
partition which dominates $\pi$.
A partition $\pi$ is said to be \emph{irreducible} 
if its interval closure
is  $\hat{1}_n$ (i.e. $1$ and $n$ are in the same block of $\pi$).
The set of irreducible noncrossing partitions of $[n]$ will be denoted  by $NC^{\mathrm{irr}}(n)$. 
 
The convex hull of a block $B\in\pi$ is the interval
$I(B)=[\min(B),\min(B)+1,\dots,\max(B)]$ and is called the \emph{padding interval} of $B$. 

A block  $B'\in\pi$  is said to be \emph{nested inside} another block $B$ if $I(B')\subset I(B)$, i.e., if $\min(B)<\min(B')\leq \max(B')< \max(B)$.

\begin{defi}\label{c2}
  For  $\rho\in \Int_{1,2}(n),$ we define the
  \emph{cyclic Boolean complement}
  $\bar \rho\in \NC^{\mathrm{irr}}(n) $ as follows:
\begin{enumerate}[1.,leftmargin=2em]
\item{} If $\rho$ has at most one brace, then $\bar\rho = \hat{1}_n$.
\item{} If \(\rho\) has \(q\geq2\) braces, write them as  
\[
\{p_1,p_1+1\},\ldots,\{p_q,p_q+1\},
\qquad p_1<\cdots<p_q.
\]Define \(\bar\rho\) to have the blocks
\[
\{p_j+1,\ldots,p_{j+1}\},\qquad 1\leq j<q,
\]together with the outer block
\[
\{p_q+1,\ldots,n\}\cup\{1,\ldots,p_1\}.
\]
\end{enumerate}
\end{defi}

\begin{Rem}
  The lattice of interval partitions $\Int(n)$ is antiisomorphic
  to the Boolean lattice of order $n-1$.
  The isomorphism maps an interval partition to its set of \emph{descents},
  that is, the end points of the blocks except the last one.
  The \emph{Boolean complement} of an interval partition is the interval
  partition with  complementary descent set
  and the cyclic Boolean complement defined above
  is the smallest irreducible noncrossing partition dominating
  the Boolean complement.
  Note that in particular $\bar \rho\in \NC^{\mathrm{irr}}(n) $ and that $\bar \rho \vee \rho= \hat{1}_n$.
\end{Rem}

\begin{figure}[ht]
	\centering

\begin{minipage}{0.8\textwidth}
		\centering
		\begin{tikzpicture}[scale=0.08,baseline,remember picture]

			\node at (-8,12) {$\rho$:};

			\draw (2,11)--(2,15.5);
			\draw (8,11)--(8,15.5);
			\draw (14,11)--(14,15.5);
			\draw (20,11)--(20,15.5);
			\draw (26,11)--(26,15.5);
			\draw (32,11)--(32,15.5);
			\draw (38,11)--(38,15.5);
			\draw (44,11)--(44,15.5);
			\draw (50,11)--(50,15.5);
			\draw (56,11)--(56,15.5);
			\draw (62,11)--(62,15.5);
			\draw (68,11)--(68,15.5);
			\draw (74,11)--(74,15.5);
			\draw (80,11)--(80,15.5);
			\draw (86,11)--(86,15.5);
			\draw (92,11)--(92,15.5);

			\draw (14,15.5)--(20,15.5);
			\draw (26,15.5)--(32,15.5);
			\draw (38,15.5)--(44,15.5);
			\draw (62,15.5)--(68,15.5);
			\draw (74,15.5)--(80,15.5);

\node at (-8,-3) {$\bar\rho$:};

			\draw (2,-5)--(2,1.5);
			\draw (8,-5)--(8,1.5);
			\draw (14,-5)--(14,1.5);

			\draw (20,-5)--(20,-.5);
			\draw (26,-5)--(26,-.5);
			\draw (32,-5)--(32,-.5);
			\draw (38,-5)--(38,-.5);
			\draw (44,-5)--(44,-.5);
			\draw (50,-5)--(50,-.5);
			\draw (56,-5)--(56,-.5);
			\draw (62,-5)--(62,-.5);
			\draw (68,-5)--(68,-.5);
			\draw (74,-5)--(74,-.5);

			\draw (80,-5)--(80,1.5);
			\draw (86,-5)--(86,1.5);
			\draw (92,-5)--(92,1.5);

			\draw (20,-.5)--(26,-.5);
			\draw (32,-.5)--(38,-.5);
			\draw (44,-.5)--(62,-.5);
			\draw (68,-.5)--(74,-.5);
			\draw (2,1.5)--(92,1.5);

		\end{tikzpicture}
	\end{minipage}

	\caption{}
	\label{fig:complement}
\end{figure}

We will denote by
$$
\uComp{\rho} = \{\pi\in\NC(n) \mid \pi\vee\rho
=\hat{1}_{ n} 
\}
$$
the set of \emph{upper semicomplements} of $\rho$.

We now introduce the main definition of this section,
which distinguishes blocks of $\pi\in \uComp{\rho}$
according to the way their endpoints are connected
by the blocks $\rho$.
\begin{defi}\label{def:multe}
  Fix a Fibonacci partition $\rho\in \Int_{1,2}(n)$, a noncrossing partition  $\pi\in \uComp{\rho} $ 
 and two distinct blocks  $B,B'\in\pi$.
 \begin{enumerate}[1.,leftmargin=2em]
\item We say that a brace $C\in \rho$ \emph{connects} the blocks $B$ and $B'$
    if both $B\cap C\ne\emptyset$ and $B'\cap C\ne\emptyset$.

   \item
A  \emph{$\rho$-1-connected block} or simply  \emph{1-connected block}  of $\pi$ is a block $B\in\pi$ such that only one  of the pairs $\{ \min(B)-1,\min(B)\}$ and $\{ \max(B),\max(B)+1\}$ is a brace in $\rho$, but not both.
\item A \emph{ 1-connected block} $ B $ of $ \pi $ is said to be \emph{inner}\footnote{Note that with this definition an inner 1-connected block could be an outer block in the usual sense.} if no 1-connected block of $ \pi$ is nested inside $ B $.
\item
    A
\emph{$2$-connected block} of $\pi$ is a block $B\in\pi$
    such that both $\{ \min(B)-1,\min(B)\}$ and $\{ \max(B),\max(B)+1\}$ are braces in $\rho$.  
   \item
    The remaining blocks of $\pi$ are called
    \emph{$0$-connected}, i.e., $B$ is 0-connected if
the blocks of $\rho$ that contain the extreme points of
      $B$ are both contained in the padding interval $I(B)$.
\end{enumerate}
\end{defi}

In short one can write this as follows. Let
\[
\ell_\rho(B)=
\begin{cases}
1,&\min(B)>1\text{ and }\{\min(B)-1,\min(B)\}\in\rho,\\
0,&\text{otherwise},
\end{cases}
\]and
\[
r_\rho(B)=
\begin{cases}
1,&\max(B)<n\text{ and }\{\max(B),\max(B)+1\}\in\rho,\\
0,&\text{otherwise}.
\end{cases}
\]Then define \(B\) to be \(j\)-connected when
\[
\ell_\rho(B)+r_\rho(B)=j,\qquad j\in\{0,1,2\}.
\]

\begin{Rem}
  It is easy to see that every $\pi\in\uComp{\rho}$ has the following properties.
  \begin{enumerate}[label=\arabic*.,leftmargin=2em]
   \item For every $\pi\in\uComp{\rho}$, each block of $\pi$ intersects at least one brace of $\rho$, except when $\rho=\hat{0}_n$; in that case, $\uComp{\rho}=\{\hat{1}_n\}$.
   \item Singletons of every $\pi\in\uComp{\rho}$ other than $\{1\}$ and $\{n\}$ are always 1-connected.
\item   $2$-connected blocks
    are bridges, each supported by two ``pillars''
  formed by braces connected to their neighbouring blocks,
  while $1$-connected blocks are ''half-bridges'' only.

  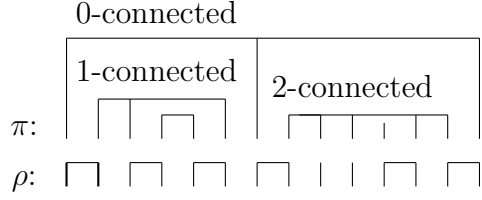
\begin{figure}[ht]
    \centering
      \begin{subfigure}[b!]{0.4\textwidth}
      \centering
      \begin{tikzpicture}[scale=0.07,baseline,remember picture]
        \node[color=white] at (31,25) {.};
        \node[color=white] at (31,-5) {.};
\draw (2,0)--(2,19);
        \draw (8,0)--(8,7.5);
        \draw (14,0)--(14,7.5);
        \draw  (20,0)--(20,4.5);
        \draw  (26,0)--(26,4.5);
        \draw (32,0)--(32,7.5);
        \draw (38,0)--(38,19);
        \draw  (44,0)--(44,4.5);
        \draw  (50,0)--(50,4.5);
        \draw  (56,0)--(56,4.5);
        \draw  (62,0)--(62,3);
\draw  (68,0)--(68,4.5);
        \draw  (74,0)--(74,4.5);
        \draw (80,0)--(80,19);
        \draw  (20,4.5)--(26,4.5);
        \draw (8,7.5)--(32,7.5);
        \draw  (44,4.5)--(50,4.5);
        \draw  (46,4.5)--(74,4.5);
        \draw (2,19)--(80,19);
\node at (19,13) {1-connected};
        \node at (56,10) {2-connected};
        \node at (19,24) {0-connected};
\node at (-6, 2){$\pi$:};

        \node at (-6,-8){$\rho$:};
        \draw[thick] (2,-9)--(2,-4.5);
        \draw[thick] (8,-9)--(8,-4.5);
        \draw (14,-9)--(14,-4.5);
        \draw (20,-9)--(20,-4.5);
        \draw (26,-9)--(26,-4.5);
        \draw (32,-9)--(32,-4.5);
        \draw (38,-9)--(38,-4.5);
        \draw (44,-9)--(44,-4.5);
        \draw(50,-9)--(50,-4.5);
        \draw(56,-9)--(56,-4.5);
\draw (62,-9)--(62,-4.5);
        \draw (68,-9)--(68,-4.5);
        \draw (74,-9)--(74,-4.5);
        \draw (80,-9)--(80,-4.5);
        \draw  (2,-4.5)--(8,-4.5);
        \draw (14,-4.5)--(20,-4.5);
        \draw (26,-4.5)--(32,-4.5);
        \draw (38,-4.5)--(44,-4.5);
        \draw (62,-4.5)--(68,-4.5);
        \draw (74,-4.5)--(80,-4.5);
      \end{tikzpicture}

\end{subfigure}
    
    \caption{0-, 1- and 2-connected blocks}
    \label{fig:innerodd}
  \end{figure}
 \item In the absence of singletons the 1-connected blocks
  are exactly the odd blocks considered in \cite{EjsmontLehner:2019:commutators}.

  \end{enumerate}
\end{Rem}

The following definition splits $\uComp\rho$ into two disjoint sets, which will be crucial for our further analysis.

\begin{defi} 
  \begin{enumerate}[1.]
   \item []
   \item 
    Fix a Fibonacci partition $\rho$.
    A noncrossing partition $\pi\in\uComp\rho$ is called
    \emph{$\rho$-1-connected}  or simply  \emph{1-connected}
    if it has  at least one 1-connected block.
    The pair $(\rho,\pi)$ is called a
    \emph{1-connected configuration}.
    The set of all  1-connected partitions $\pi\in\uComp\rho$ will be denoted by 
    $\oneconn{\rho}$ and we denote the set of
    all 1-connected configurations by
    $$
    \oneconn{n}= \{ ( \rho, \pi):\   \rho\in \Int_{1,2}(n),
    \pi\in\oneconn{\rho}\}.
    $$
 \item 
       A noncrossing partition $\pi\in\uComp\rho$ is called
    \emph{purely 2-connected} if all its blocks are 2-connected or 0-connected; the set of all purely 2-connected partitions from $\pi\in\uComp\rho$ will be denoted by 
    $\twoconn{\rho}$ and and we put
    $$
    \twoconn{n}= \{ ( \rho, \pi):\   \rho\in \Int_{1,2}(n), \pi\in\twoconn{\rho}\}.
    $$
  \end{enumerate}
\end{defi}

\begin{defi}
  Let $\rho\in \Int_{1,2}(n)$ and 
  $\pi\in \oneconn{\rho}$.

  \begin{enumerate}[1.,leftmargin=2em]
   \item  If  $B$ is a 1-connected block of $ \pi $ and $I(B)$ is its padding interval, then, by definition
    there exist a unique  brace $C\in \rho$
    such that  $I(B)\cap C$ contains exactly one element.
    The brace $C$
    will be called the \emph{pivot brace} of $B$
    and the unique point in its intersection with $B$
    will be called the \emph{pivot element} of $B$.
     The block $B$ will be called \emph{left-connected} if the pivot brace is equal to $\{\min(B)-1,\min(B)\}$
    and \emph{right-connected} if the pivot brace is equal to  $ \{\max(B),\max(B)+1\}$ 
   \item The pivot brace of the leftmost inner
    1-connected block
    of $\pi$ is called the  \emph{pivot brace of $\pi$} and is denoted by
    $ \SpPair_\pi$.
    In our figures the pivot brace will be highlighted by thick lines (see Figure~\ref{fig:nest}).

    \begin{figure}[ht]
      \begin{subfigure}[b!]{0.4\textwidth}
        \centering
        \begin{tikzpicture}[scale=0.07,baseline,remember picture]
          \node[color=white] at (31,27.675) {.};
          \node[color=white] at (31,-5) {.};
          
          \draw (2,0)--(2,7.5);
          \draw  (8,0)--(8,4.5);
          \draw  (14,0)--(14,4.5);
          \draw  (8,4.5)--(14,4.5);
          \draw (20,0)--(20,7.5);
          \draw (26,0)--(26,7.5);
          \draw  (32,0)--(32,4.5);
          \draw  (38,0)--(38,4.5);
          \draw  (32,4.5)--(38,4.5);
          \draw (44,0)--(44,7.5);
          \draw (50,0)--(50,7.5);
          \draw (56,0)--(56,7.5);
          \draw (62,0)--(62,7.5);
          \draw (68,0)--(68,4.5);
          \draw  (74,0)--(74,4.5);
          \draw  (68,4.5)--(74,4.5);
          \draw (80,0)--(80,7.5);
          \draw (2,7.5)--(50,7.5);
          \draw (56,7.5)--(80,7.5);
          
          \node at (31,13) {Right-connected};
          \node at (54.5,-13) {$\SpPair_\pi$};

          \draw (2,-9)--(2,-4.5);
          \draw (8,-9)--(8,-4.5);
          \draw (14,-9)--(14,-4.5);
          \draw (20,-9)--(20,-4.5);
          \draw (26,-9)--(26,-4.5);
          \draw (32,-9)--(32,-4.5);
          \draw (38,-9)--(38,-4.5);
          \draw (44,-9)--(44,-4.5);
          \draw[thick] (50,-9)--(50,-4.5);
          \draw[thick] (56,-9)--(56,-4.5);
          \draw[thick] (50,-4.5)--(56,-4.5);
          \draw (62,-9)--(62,-4.5);
          \draw (68,-9)--(68,-4.5);
          \draw (74,-9)--(74,-4.5);
          \draw (80,-9)--(80,-4.5);
          \draw (2,-4.5)--(8,-4.5);
          \draw (14,-4.5)--(20,-4.5);
          \draw (26,-4.5)--(32,-4.5);
          \draw (38,-4.5)--(44,-4.5);
\draw (74,-4.5)--(80,-4.5);
        \end{tikzpicture}

\end{subfigure}
      ~
      \begin{subfigure}[b!]{0.4\textwidth}
        \centering
        \begin{tikzpicture}[scale=0.07,baseline,remember picture]
          \node[color=white] at (31,25) {.};
          \node[color=white] at (31,-5) {.};

          \draw (2,0)--(2,10);
          \draw (8,0)--(8,7.5);
          \draw (14,0)--(14,7.5);
          \draw  (20,0)--(20,4.5);
          \draw  (26,0)--(26,4.5);
          \draw (32,0)--(32,7.5);
          \draw (38,0)--(38,10);
          \draw  (44,0)--(44,4.5);
          \draw  (50,0)--(50,4.5);
          \draw (56,0)--(56,10);
          \draw (62,0)--(62,10);
          \draw  (68,0)--(68,4.5);
          \draw  (74,0)--(74,4.5);
          \draw (80,0)--(80,10);
          \draw  (20,4.5)--(26,4.5);
          \draw (8,7.5)--(32,7.5);
          \draw  (44,4.5)--(50,4.5);
          \draw  (68,4.5)--(74,4.5);
          \draw (2,10)--(80,10);
          
          \node at (29,14) {Left-connected};
          \node at (6.5,-13) {$\SpPair_\pi$};

          \draw[thick] (2,-9)--(2,-4.5);
          \draw[thick] (8,-9)--(8,-4.5);
          \draw (14,-9)--(14,-4.5);
          \draw (20,-9)--(20,-4.5);
          \draw (26,-9)--(26,-4.5);
          \draw (32,-9)--(32,-4.5);
          \draw (38,-9)--(38,-4.5);
          \draw (44,-9)--(44,-4.5);
          \draw(50,-9)--(50,-4.5);
          \draw(56,-9)--(56,-4.5);
\draw (62,-9)--(62,-4.5);
          \draw (68,-9)--(68,-4.5);
          \draw (74,-9)--(74,-4.5);
          \draw (80,-9)--(80,-4.5);
          \draw[thick] (2,-4.5)--(8,-4.5);
          \draw (14,-4.5)--(20,-4.5);
          \draw (26,-4.5)--(32,-4.5);
          \draw (38,-4.5)--(44,-4.5);
\draw (74,-4.5)--(80,-4.5);
        \end{tikzpicture}

\end{subfigure}
      
      \caption{Examples of right- and left-connected blocks} \label{fig:nest}
    \end{figure}
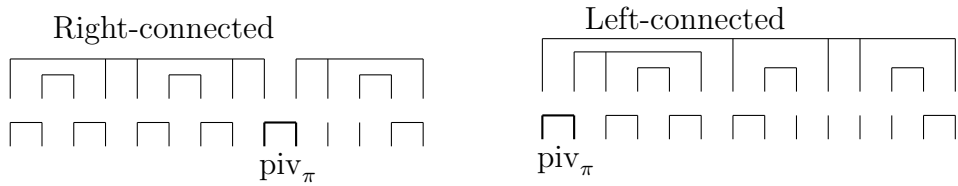
  \end{enumerate}
\end{defi}

\subsection{Lemmata}
In this subsection characterize the different types of partitions
introduced above.
\begin{lemm}
  \label{lem:0-connected}
  Fix a Fibonacci partition  $\rho$.
  A partition $\pi\in\uComp{\rho}$ is irreducible if and only if it has a
  0-connected block. In this case
  the 0-connected block is the unique outer block.
\end{lemm}
\begin{proof}
If $ \pi $ is irreducible, then its unique outer block is obviously 0-connected.

  Conversely,
  assume that  $B $ is a 0-connected block of the partition $ \pi $ from $\uComp{\rho}$.
  Denote by $\pi_0$ the noncrossing partition with at most two blocks, consisting of
  the padding interval $I(B)$ and its complement $ D =[n] \setminus I(B) $.
  
  First, note that $ \pi \leq \pi_0 $, since all blocks of $\pi$ inside $B$ are
  contained in $ I(B)$, while all other blocks of $\pi$ are contained in $ D $.
  
  Next, note that $\rho\leq\pi_0$, since, by the definition of the 0-connected blocks, the blocks of $\rho$ containing the endpoints of $B$  are
    contained in the padding interval $I(B)$, and so are all blocks
    in between,
    while
    the other blocks of $\rho$ are contained in $D$.
  
 It follows that
    $\pi\vee\rho\leq \pi_0$, which contradicts $\pi\vee\rho =\hat{1}_{ n}$ unless $ D = \emptyset$, i.e. $I(B) = [n]$, so $ \pi $ is irreducible.
\end{proof}

\begin{lemm} \label{lemtwoinnerodd}
  Fix a Fibonacci partition $\rho\in \Int_{1,2}(n)$.
  If a partition $\pi\in\oneconn{\rho} $
  has exactly one 1-connected block then it  has a unique  0-connected block.  \end{lemm}

\begin{proof}
  Observe that the blocks containing the extreme points $1$ and $n$ are either
  $0$-connected or $1$-connected.
  From the previous lemma it follows that if one of them is $1$-connected, the other
  one is $1$-connected as well and in this case there are at least two
  $1$-connected blocks, contradicting the assumption that $\pi$ has exactly one
  1-connected block.
  Thus $\pi$ is irreducible and its outer block is 0-connected.
\end{proof}

\begin{lemm}\label{lemmatwoone}
Fix a Fibonacci partition $\rho\in \Int_{1,2}(n)$ and 
  let   $\pi\in\oneconn{\rho}$.
  If $\pi$ has  two   1-connected blocks $ B $ and $B^\prime$ such that $ \min(B^\prime) = \max(B)+1$ and $ \{ \max(B), \min(B^\prime)\}$ is a brace in $\rho$ connecting $ B$ and $ B^\prime$, then the union of the padding intervals of $ B $ and $ B^\prime$ is $ [ n]$.  
\end{lemm}

\begin{proof}

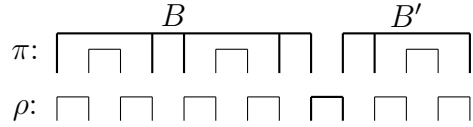
\begin{figure}[h]
  \centering
  \begin{tikzpicture}[scale=0.07]
    \node at (-4, 4){$\pi$:};
    
      \node at (24, 11){$B$};
      
      \node at (68, 11){$B^\prime$};
  
    \draw[thline] (2,0)--(2,7.5);
    \draw[thline] (20,0)--(20,7.5);
    \draw[thline] (26,0)--(26,7.5);
    \draw[thline] (44,0)--(44,7.5);
    \draw[thline] (50,0)--(50,7.5);
    \draw[thline] (2,7.5)--(50,7.5);
    
    \draw (8,0)--(8,4.5);
    \draw (14,0)--(14,4.5);

    \draw (32,0)--(32,4.5);
    \draw (38,0)--(38,4.5);

    \draw[thline] (56,0)--(56,7.5);
    \draw[thline] (62,0)--(62,7.5);
    \draw[thline] (80,0)--(80,7.5);
    \draw[thline] (56,7.5)--(80,7.5);
    
    \draw (68,0)--(68,4.5);
    \draw (74,0)--(74,4.5);
    \draw (8,4.5)--(14,4.5);
    \draw (32,4.5)--(38,4.5);
    \draw (68,4.5)--(74,4.5);
    
    \node at(-4, -7){$\rho$:};
    
    \draw (2,-9)--(2,-4.5);
    \draw (8,-9)--(8,-4.5);
    \draw (14,-9)--(14,-4.5);
    \draw (20,-9)--(20,-4.5);
    \draw (26,-9)--(26,-4.5);
    \draw (32,-9)--(32,-4.5);
    \draw (38,-9)--(38,-4.5);
    \draw (44,-9)--(44,-4.5);
    \draw[thline] (50,-9)--(50,-4.5);
    \draw[thline] (56,-9)--(56,-4.5);
    \draw (62,-9)--(62,-4.5);
    \draw (68,-9)--(68,-4.5);
    \draw (74,-9)--(74,-4.5);
    \draw (80,-9)--(80,-4.5);
    \draw (2,-4.5)--(8,-4.5);
    \draw (14,-4.5)--(20,-4.5);
    \draw (26,-4.5)--(32,-4.5);
    \draw (38,-4.5)--(44,-4.5);
    \draw[thline] (50,-4.5)--(56,-4.5);
    \draw (62,-4.5)--(68,-4.5);
    \draw (74,-4.5)--(80,-4.5);
  \end{tikzpicture}
  
  \caption{Two adjacent 1-connected  blocks.}
  \label{fig:adjacentinnerodd}
\end{figure}

  Let $ I = I(B)\cup I(B')$ and  let $ \pi_0 $ be the partition with blocks $ I $ and $ [n]\setminus I $. Note that, according to Definition~\ref{def:multe},  $ I $  is an interval which is a union of blocks from $\rho$.  It follows that both $\pi\leq\pi_0$ and $\rho\leq\pi_0$,
and consequently $\hat{1}_n=\pi\vee\rho\leq\pi_0$.
  Hence  $I=[n]$.
  \end{proof}

\begin{lemm}\label{prop:tylkowieksze}
  Fix a Fibonacci partition $ \rho\in \Int_{1,2}(n)$. Then
$\pi\in \CompOdd{\rho} $ if and only if  
    $  \pi \geq \bar \rho$.
\end{lemm}
\begin{proof}
  Assume that $ \pi \geq \bar \rho$. Now by construction, $\bar\rho$ is irreducible (hence so is $\pi$), and
  by definition all inner blocks of $\bar\rho$ are $\rho$-2-connected and  so are the their unions, in particular the inner blocks of $\pi$. 

  For the converse,
 assume that $ \pi \in \CompOdd{\rho} $. If $\rho$ has at most one brace, then
 $ \pi = \hat{1}_n$ and the conclusion follows, so we can assume that $ \rho$
 has at least two braces. We have to show that every block of $\bar \rho$ is
 contained in a block of $\pi$.
 
 \noindent
 Let $\{a,\ldots,b\}$ be a block of $\bar \rho$,
 then  $\rho$ contains the braces $\{a-1, a\}$ and $\{b, b+1\}$
 and all elements between them form singletons of $\rho$.

 \noindent
 Let $A$ be the block of $\pi$  containing $a$
 and $B$ the block containing $b$.
 By Lemma~\ref{lem:0-connected} at most one of them is 0-connected.
 Assume wlog that $A$ is 2-connected
and 
 suppose that $A\ne B$.  
 Since $A$ is 2-connected,  its right end point $a'$ is  the left element of a
 brace in $\rho$  and thus $a'>b+1$.
 It follows that $B$ is nested inside $A$ and therefore not
 an outer block and we infer from Lemma~\ref{lem:0-connected} that $B$ is
 2-connected.
 Therefore left end point  $b'$ of $B$
 is the right element of another brace of $\rho$ and thus $b'<a-1$.
 Therefore $b'<a<b<a'$ is a crossing of the block $A$ and $B$,
 contradicting the assumption that $\pi$ is noncrossing.
 It follows that $a$ and $b$ lie in the same block of $\pi$
 and the singletons of $\rho$ between $a$ and $b$ must all be elements of
 this block as well.

\end{proof}

\section{An involution on  the set of $1$-connected partitions}
\label{sec:involution}

In this section we shall construct an involution which will be used (in Section~\ref{sec:cumulants}) to show the cancellation of the contribution of 1-connected partitions in formula \eqref{twr:produktargumentow} when applied to polynomials as in \eqref{eqn:1.1} for skew symmetric matrices $\big[ a_{i, j} \big]_{i, j=1}^n $.
	
\begin{defi} \label{lemtwoinnerodd}
  \label{def:flippartition}
  A pair $(\rho,\pi)\in \oneconn{n}$ is said to be
  \begin{enumerate}[label=\arabic*.,leftmargin=2em]
   \item 
    \textit{type I} if $\pi$ has two inner 1-connected blocks with a common pivot brace in $\rho$;
  
   \item  \textit{type IIa} if $\pi$  has exactly one 1-connected block with pivot brace $\{1, 2\}\in\rho$;
   \item  \textit{type IIb} if $\pi$ has exactly one 1-connected block with pivot brace $\{n-1, n\}\in\rho$;
 
   \item
   \textit{type IIIa} if $\pi$ is not type I and II and the leftmost inner 1-connected block is left connected;
   \item
   \textit{type IIIb}  if $\pi$ is not type I and II and the leftmost inner 1-connected block is right connected.
  \end{enumerate}
We have the disjoint union $\oneconn{n}=\oneconntype{n}{I}\sqcup \oneconntype{n}{IIa}\sqcup \oneconntype{n}{IIb}\sqcup \oneconntype{n}{IIIa}\sqcup \oneconntype{n}{IIIb}$.
\end{defi}

Examples of different types of partitions are shown in Figure~\ref{fig:flippartitions} below.
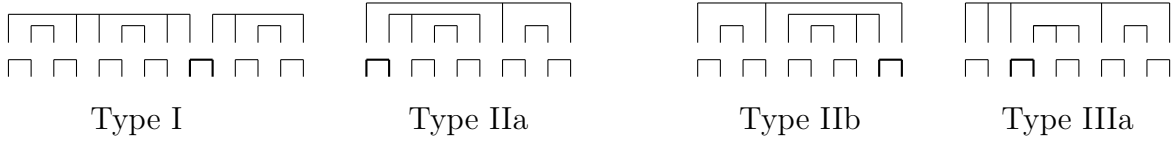
\begin{figure}[h]
  \begin{subfigure}[t]{0.2\textwidth}
    \centering
    \begin{tikzpicture}[scale=0.05]
      \draw (2,0)--(2,7.5);
      \draw (8,0)--(8,4.5);
      \draw (14,0)--(14,4.5);
      \draw (20,0)--(20,7.5);
      \draw (26,0)--(26,7.5);
      \draw (32,0)--(32,4.5);
      \draw (38,0)--(38,4.5);
      \draw (44,0)--(44,7.5);
      \draw (50,0)--(50,7.5);
      \draw (56,0)--(56,7.5);
      \draw (62,0)--(62,7.5);
      \draw (68,0)--(68,4.5);
      \draw (74,0)--(74,4.5);
      \draw (80,0)--(80,7.5);
      \draw (8,4.5)--(14,4.5);
      \draw (32,4.5)--(38,4.5);
      \draw (2,7.5)--(50,7.5);
      \draw (68,4.5)--(74,4.5);
      \draw (56,7.5)--(80,7.5);

      \draw (2,-9)--(2,-4.5);
      \draw (8,-9)--(8,-4.5);
      \draw (14,-9)--(14,-4.5);
      \draw (20,-9)--(20,-4.5);
      \draw (26,-9)--(26,-4.5);
      \draw (32,-9)--(32,-4.5);
      \draw (38,-9)--(38,-4.5);
      \draw (44,-9)--(44,-4.5);
      \draw[thline] (50,-9)--(50,-4.5);
      \draw[thline] (56,-9)--(56,-4.5);
      \draw (62,-9)--(62,-4.5);
      \draw (68,-9)--(68,-4.5);
      \draw (74,-9)--(74,-4.5);
      \draw (80,-9)--(80,-4.5);
      \draw (2,-4.5)--(8,-4.5);
      \draw (14,-4.5)--(20,-4.5);
      \draw (26,-4.5)--(32,-4.5);
      \draw (38,-4.5)--(44,-4.5);
      \draw[thline] (50,-4.5)--(56,-4.5);
      \draw (62,-4.5)--(68,-4.5);
      \draw (74,-4.5)--(80,-4.5);
    \end{tikzpicture} 
    \vskip1ex
    
    Type I
  \end{subfigure}
  ~
  \begin{subfigure}[t]{0.3\textwidth}
    \centering
    \begin{tikzpicture}[scale=0.05]
      \draw (2,0)--(2,10.5);
      \draw (8,0)--(8,7.5);
      \draw (14,0)--(14,7.5);
      \draw (20,0)--(20,4.5);
      \draw (26,0)--(26,4.5);
      \draw (32,0)--(32,7.5);
      \draw (38,0)--(38,10.5);
      \draw (44,0)--(44,4.5);
      \draw (50,0)--(50,4.5);
      \draw (56,0)--(56,10.5);
      \draw (20,4.5)--(26,4.5);
      \draw (8,7.5)--(32,7.5);
      \draw (44,4.5)--(50,4.5);
      \draw (2,10.5)--(56,10.5);

      \draw[thline] (2,-9)--(2,-4.5);
      \draw[thline] (8,-9)--(8,-4.5);
      \draw (14,-9)--(14,-4.5);
      \draw (20,-9)--(20,-4.5);
      \draw (26,-9)--(26,-4.5);
      \draw (32,-9)--(32,-4.5);
      \draw (38,-9)--(38,-4.5);
      \draw (44,-9)--(44,-4.5);
      \draw (50,-9)--(50,-4.5);
      \draw (56,-9)--(56,-4.5);
      \draw[thline] (2,-4.5)--(8,-4.5);
      \draw (14,-4.5)--(20,-4.5);
      \draw (26,-4.5)--(32,-4.5);
      \draw (38,-4.5)--(44,-4.5);
      \draw (50,-4.5)--(56,-4.5);
    \end{tikzpicture}

    \vskip1ex
    
    Type IIa
  \end{subfigure}
  ~
  \begin{subfigure}[t]{0.2\textwidth}
    \centering
    \begin{tikzpicture}[scale = 0.05]
      \draw (2,0)--(2,10.5);
      \draw (8,0)--(8,4.5);
      \draw (14,0)--(14,4.5);
      \draw (20,0)--(20,10.5);
      \draw (26,0)--(26,7.5);
      \draw (32,0)--(32,4.5);
      \draw (38,0)--(38,4.5);
      \draw (44,0)--(44,7.5);
      \draw (50,0)--(50,7.5);
      \draw (56,0)--(56,10.5);
      \draw (8,4.5)--(14,4.5);
      \draw (32,4.5)--(38,4.5);
      \draw (26,7.5)--(50,7.5);
      \draw (2,10.5)--(56,10.5);

      \draw (2,-9)--(2,-4.5);
      \draw (8,-9)--(8,-4.5);
      \draw (14,-9)--(14,-4.5);
      \draw (20,-9)--(20,-4.5);
      \draw (26,-9)--(26,-4.5);
      \draw (32,-9)--(32,-4.5);
      \draw (38,-9)--(38,-4.5);
      \draw (44,-9)--(44,-4.5);
      \draw[thline] (50,-9)--(50,-4.5);
      \draw[thline] (56,-9)--(56,-4.5);
      \draw (2,-4.5)--(8,-4.5);
      \draw (14,-4.5)--(20,-4.5);
      \draw (26,-4.5)--(32,-4.5);
      \draw (38,-4.5)--(44,-4.5);
      \draw[thline] (50,-4.5)--(56,-4.5);
    \end{tikzpicture} 
    
    \vskip1ex
    
    Type IIb
  \end{subfigure}
  ~
  \begin{subfigure}[t]{0.2\textwidth}
    \centering
    \begin{tikzpicture}[scale=0.05]
      \draw (2,0)--(2,10.5);
      \draw (8,0)--(8,10.5);
      \draw (14,0)--(14,10.5);
      \draw (20,0)--(20,4.5);
      \draw (26,0)--(26,4.5);
      \draw (32,0)--(32,4.5);
      \draw (38,0)--(38,10.5);
      \draw (44,0)--(44,4.5);
      \draw (50,0)--(50,4.5);
      \draw (56,0)--(56,10.5);
      \draw (20,4.5)--(26,4.5);
      \draw (20,4.5)--(32,4.5);
      \draw (44,4.5)--(50,4.5);
      \draw (2,10.5)--(56,10.5);

      \draw (2,-9)--(2,-4.5);
      \draw (8,-9)--(8,-4.5);
      \draw [thline] (14,-9)--(14,-4.5);
      \draw [thline](20,-9)--(20,-4.5);
      \draw (26,-9)--(26,-4.5);
      \draw (32,-9)--(32,-4.5);
      \draw (38,-9)--(38,-4.5);
      \draw (44,-9)--(44,-4.5);
      \draw (50,-9)--(50,-4.5);
      \draw (56,-9)--(56,-4.5);
      \draw (2,-4.5)--(8,-4.5);
      \draw [thline](14,-4.5)--(20,-4.5);
      \draw (26,-4.5)--(32,-4.5);
      \draw (38,-4.5)--(44,-4.5);
      \draw (50,-4.5)--(56,-4.5);
    \end{tikzpicture}

    \vskip1ex
    
    Type IIIa
  \end{subfigure}
  \caption{Examples of type I, II and III partitions.}
  \label{fig:flippartitions}
\end{figure}

\label{ssec:involution}
In the next paragraphs we shall define an involution $ \Phi : \oneconn{n} \rightarrow \oneconn{n} $. 
To this end we shall construct for each $ (\rho, \pi) \in \oneconn{n}$ a certain permutation
$ \sigma $ on $ [n ]$ such that $ (\sigma\cdot \rho, \sigma \cdot \pi )\in
\oneconn{n}$; moreover, the block structure and type is preserved.
Intuitively, the involution follows  the patterns layed out in
Figures~\ref{fig:involutiontypeI},~\ref{fig:involutiontypeII} and~\ref{fig:involutiontypeIII}.

\begin{enumerate}[(I),leftmargin=2em]
 \item 
 A partition of type \textit{I} is mapped to another partition of type I by
  flipping the left and right subpartitions, according to the scheme
  shown in Figures~\ref{fig:involutiontypeI}.
 \begin{figure}[h]
    \begin{subfigure}[c]{0.35\textwidth}
      \centering
      \begin{tikzpicture}[scale=0.5]
        \path [afill,thick] (0,0) rectangle (5.2,1.2);
        \path [bfill,thick] (8.0,0) rectangle (13.2,1.2);
        \draw [thick] (5.2,1.2)--(6.3,1.2)--(6.3,0);
        \draw [thick] (8.0,1.2)--(6.9,1.2)--(6.9,0);
        
        \draw[white,fill=white] (2.5,0.6) circle (0.5); 
        \node at (2.5,0.6) {$A$};
        
        \draw[white,fill=white] (10.5,0.6) circle (0.5); 
        \node at (10.5,0.6) {$B$};
        \node at (6.6,-0.5){$\circ \ \bullet $};
        
        \node[color=white] at (1,2) {.};
      \end{tikzpicture} 
    \end{subfigure}
    $\qquad\qquad\longrightarrow\qquad$
    \begin{subfigure}[c]{0.35\textwidth}
      \centering
      \begin{tikzpicture}[scale=0.5]
        \path [bfill,thick] (0,0) rectangle (5.2,1.2);
        \path [afill,thick] (8.0,0) rectangle (13.2,1.2);
        \draw [thick] (5.2,1.2)--(6.3,1.2)--(6.3,0);
        \draw [thick] (8.0,1.2)--(6.9,1.2)--(6.9,0);
        
        \draw[white,fill=white] (2.5,0.6) circle (0.5); 
        \node at (2.5,0.6) {$B$};
        
        \draw[white,fill=white] (10.5,0.6) circle (0.5); 
        \node at (10.5,0.6) {$A$};
        \node at (6.6,-0.5){$ \bullet   \ \circ $};
        
        \node[color=white] at (1,2) {.};
      \end{tikzpicture} 
    \end{subfigure}
    
    \caption{Involution on $\oneconntype{n}{I}$.} \label{fig:involutiontypeI}
  \end{figure}

 Suppose that the pair $(\rho, \pi) \in \oneconn{n} $ is of type \textit{I}. According to Lemma~\ref{lemmatwoone}, there exist some $ k  \in [ n-1]$ such that the padding intervals of the leftmost, respectively  rightmost, 1-connected inner block are $\{1, 2, \dots, k \}$, respectively $\{k+1, \dots, n \}$ and $\{ k, k+1\} $ is a brace in $ \rho$. The involution
  $\Phi$ maps the pair $(\rho, \pi)$ to the pair $(\sigma_{I,n,k} \cdot \rho, \sigma_{I,n,k}\cdot \pi)$, where  the permutation $\sigma_{I,n,k}:[n] \rightarrow [n]$ is given by
   \begin{align*}
    \sigma_{I,n,k}(i)
    &=
      \begin{cases}
        i+n-k+1 & \text{for $1\leq i\leq k-1$ }\\
        n-i +1  & \text{for $i=k$ or $i=k+1$}\\
i-( k+1) & \text{for $k+2\leq i\leq n$ }
      \end{cases}
                   ;
  \end{align*}
i.e., $\sigma_{I, n, k}$ exchanges the intervals $[ 1, 2,\dots, k-1]$ and $[k+2, k+3, \dots, n]$ and the points $k $ and $k+1$. Since both these intervals are unions of blocks of $ \rho$ and  $ \{k, k+1\}$ is a brace in $ \rho$ it follows that 
 $ \sigma_{I, n, k} \cdot \rho \in \Int_{1,2}(n)$. Moreover, $ \sigma_{I, n, k} \cdot \pi $ is in $ \oneconn{\sigma_{I, n, k} \cdot\rho}$ since $ \sigma_{I, n, k}\cdot \rho \vee \sigma_{I, n, k}\cdot \pi = \sigma_{I, n, k}\cdot (\rho \vee \pi )$
 and 
 $\sigma_{I, n, k} \cdot \pi $ is of type \textit{I} (see Figure~\ref{fig:exinvolutiontypeI} for a specific example).

 Finally, note that $  \sigma_{I,n,k} =\sigma_{I,n,n-k}^{-1}$, hence $ \Phi$ is an involution of $\oneconntype{n}{I}$.
   \begin{figure}[h]
    \begin{subfigure}[c]{0.3\textwidth}
      \centering
      \begin{tikzpicture}[scale=0.04]
        \draw[style=aline] (2,0)--(2,7.5);
        \draw[style=aline] (8,0)--(8,4.5);
        \draw[style=aline] (14,0)--(14,4.5);
        \draw[style=aline] (20,0)--(20,4.5);
        \draw[style=aline] (26,0)--(26,4.5);
        \draw[style=aline] (32,0)--(32,4.5);
        \draw[style=aline] (38,0)--(38,4.5);
        \draw[style=aline] (44,0)--(44,7.5);
        \draw[style=aline] (2,7.5)--(44,7.5);
        \draw[style=aline] (8,4.5)--(14,4.5);
        \draw[style=aline] (20,4.5)--(38,4.5);
        
        \draw (50,0)--(50,7.5)--(44,7.5);
        \draw (56,0)--(56,7.5)--(62,7.5);
        
        \draw[style=bline] (62,0)--(62,7.5);
        \draw[style=bline] (68,0)--(68,4.5);
        \draw[style=bline] (74,0)--(74,4.5);
        \draw[style=bline] (80,0)--(80,4.5);
        \draw[style=bline] (86,0)--(86,4.5);
        \draw[style=bline] (92,0)--(92,4.5);
        \draw[style=bline] (98,0)--(98,4.5);
        \draw[style=bline] (104,0)--(104,4.5);
        \draw[style=bline] (110,0)--(110,4.5);
        \draw[style=bline] (116,0)--(116,7.5);
        \draw[style=bline] (68,4.5)--(74,4.5);
        \draw[style=bline] (80,4.5)--(86,4.5);
        \draw[style=bline] (92,4.5)--(110,4.5);
        \draw[style=bline] (62,7.5)--(116,7.5);

        \draw (2,-9)--(2,-4.5);
        \draw (8,-9)--(8,-4.5);
        \draw (14,-9)--(14,-4.5);
        \draw (20,-9)--(20,-4.5);
        \draw (26,-9)--(26,-4.5);
        \draw (32,-9)--(32,-4.5);
        \draw (38,-9)--(38,-4.5);
        \draw (44,-9)--(44,-4.5);
        \draw[thick] (50,-9)--(50,-4.5);
        \draw[thick] (56,-9)--(56,-4.5);
        \draw (62,-9)--(62,-4.5);
        \draw (68,-9)--(68,-4.5);
        \draw (74,-9)--(74,-4.5);
        \draw (80,-9)--(80,-4.5);
        \draw (86,-9)--(86,-4.5);
        \draw (92,-9)--(92,-4.5);
        \draw (98,-9)--(98,-4.5);
        \draw (104,-9)--(104,-4.5);
        \draw (110,-9)--(110,-4.5);
        \draw (116,-9)--(116,-4.5);
        \draw (2,-4.5)--(8,-4.5);
        \draw (14,-4.5)--(20,-4.5);
        \draw (26,-4.5)--(32,-4.5);
        \draw (38,-4.5)--(44,-4.5);
        \draw[thick] (50,-4.5)--(56,-4.5);
        \draw (62,-4.5)--(68,-4.5);
        \draw (74,-4.5)--(80,-4.5);
        \draw (86,-4.5)--(92,-4.5);
\draw (110,-4.5)--(116,-4.5);
        
\node at (65,-20) {$\rho$};
        \node at (65,20) {$\pi$};
      \end{tikzpicture} 
    \end{subfigure}
    $\longrightarrow$
    \begin{subfigure}[c]{0.3\textwidth}
      \centering
      \begin{tikzpicture}[scale=0.04]
        \draw[style=bline] (2,0)--(2,7.5);
        \draw[style=bline] (8,0)--(8,4.5);
        \draw[style=bline] (14,0)--(14,4.5);
        \draw[style=bline] (20,0)--(20,4.5);
        \draw[style=bline] (26,0)--(26,4.5);
        \draw[style=bline] (32,0)--(32,4.5);
        \draw[style=bline] (38,0)--(38,4.5);
        \draw[style=bline] (44,0)--(44,4.5);
        \draw[style=bline] (50,0)--(50,4.5);
        \draw[style=bline] (56,0)--(56,7.5);

        \draw (62,0)--(62,7.5)--(56,7.5);
        \draw (68,0)--(68,7.5)--(74,7.5);
        
        \draw[style=aline] (74,0)--(74,7.5);
        \draw[style=aline] (80,0)--(80,4.5);
        \draw[style=aline] (86,0)--(86,4.5);
        \draw[style=aline] (92,0)--(92,4.5);
        \draw[style=aline] (98,0)--(98,4.5);
        \draw[style=aline] (104,0)--(104,4.5);
        \draw[style=aline] (110,0)--(110,4.5);
        \draw[style=aline] (116,0)--(116,7.5);
        
        \draw[style=bline] (8,4.5)--(14,4.5);
        \draw[style=bline] (20,4.5)--(26,4.5);
        \draw[style=bline] (32,4.5)--(50,4.5);
        \draw[style=bline] (2,7.5)--(56,7.5);
        
        \draw[style=aline] (80,4.5)--(86,4.5);
        \draw[style=aline] (92,4.5)--(110,4.5);
        \draw[style=aline] (74,7.5)--(116,7.5);
        
        \draw (2,-9)--(2,-4.5);
        \draw (8,-9)--(8,-4.5);
        \draw (14,-9)--(14,-4.5);
        \draw (20,-9)--(20,-4.5);
        \draw (26,-9)--(26,-4.5);
        \draw (32,-9)--(32,-4.5);
        \draw (38,-9)--(38,-4.5);
        \draw (44,-9)--(44,-4.5);
        \draw (50,-9)--(50,-4.5);
        \draw (56,-9)--(56,-4.5);
        \draw[thick] (62,-9)--(62,-4.5);
        \draw[thick] (68,-9)--(68,-4.5);
        \draw (74,-9)--(74,-4.5);
        \draw (80,-9)--(80,-4.5);
        \draw (86,-9)--(86,-4.5);
        \draw (92,-9)--(92,-4.5);
        \draw (98,-9)--(98,-4.5);
        \draw (104,-9)--(104,-4.5);
        \draw (110,-9)--(110,-4.5);
        \draw (116,-9)--(116,-4.5);
        \draw (2,-4.5)--(8,-4.5);
        \draw (14,-4.5)--(20,-4.5);
        \draw (26,-4.5)--(32,-4.5);
\draw (50,-4.5)--(56,-4.5);
        \draw[thick] (62,-4.5)--(68,-4.5);
        \draw (74,-4.5)--(80,-4.5);
        \draw (86,-4.5)--(92,-4.5);
        \draw (98,-4.5)--(104,-4.5);
        \draw (110,-4.5)--(116,-4.5);
        
\node at (65,-20) {$\rho'$};
        \node at (65,20) {$\pi'$};
      \end{tikzpicture}

    \end{subfigure}
    \vskip1ex
    \caption{An example of the involution of partitions of type I.} \label{fig:exinvolutiontypeI}
  \end{figure}
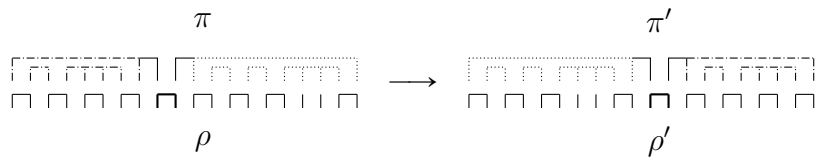

 \item type \textit{IIa} is mapped to type \textit{IIb} and vice versa, see
  Figure~\ref{fig:involutiontypeII}.

  \begin{figure}[h]
    \begin{subfigure}[c]{0.35\textwidth}
      \centering
      \begin{tikzpicture}[scale=0.5]
        \draw [white] (0,3.5) circle (0);
        \path [afill,thick] (1.8,0) rectangle (7.0,1.2);
        \path [bfill,thick] (7.6,0) rectangle (12.8,1.8);
        \draw [thick] (0.25,0)--(0.25,1.8)--(7.6,1.8);
        \draw [thick] (0.9,0)--(0.9,1.2)--(1.8,1.2);
        
        \draw[white,fill=white] (4.5,0.6) circle (0.5); 
        \node at (4.5,0.6) {$A$};
        
        \draw[white,fill=white] (10.3,0.9) circle (0.5); 
        \node at (10.3,0.9) {$B$};
        \node at (0.6,-0.5){$\circ \ \bullet$};
      \end{tikzpicture} 
    \end{subfigure}
    $\qquad\qquad\longleftrightarrow\qquad$
    \begin{subfigure}[c]{0.35\textwidth}
      \centering
      \begin{tikzpicture}[scale=0.5]
        \draw [white] (0,3.5) circle (0);
        \path [bfill,thick] (0,0) rectangle (5.2,1.8);
        \path [afill,thick] (5.8,0) rectangle (11.0,1.2);
        \draw [thick] (12.55,0)--(12.55,1.8)--(5.2,1.8);
        \draw [thick] (11.9,0)--(11.9,1.2)--(11.0,1.2);
        
        \draw[white,fill=white] (2.6,0.9) circle (0.5); 
        \node at (2.6,0.85) {$B$};
        
        \draw[white,fill=white] (8.5,0.6) circle (0.5); 
        \node at (8.5,0.6) {$A$};
        \node at (12.3,-0.5){$\bullet \  \circ$};
      \end{tikzpicture} 
    \end{subfigure}
    \caption{Involution on $\oneconntype{n}{II}$.} \label{fig:involutiontypeII}
  \end{figure}

First assume that $(\rho, \pi) \in \oneconntype{n}{IIa}$. By Definition~\ref{lemtwoinnerodd} there exists $ k \in \{2, 3, \dots, n-1\}$ such that the padding 
 interval of the unique
  1-connected block is $\{2,3,\dots ,k\}$.
In this case we define 
$\Phi((\rho, \pi)) = (\sigma_{IIa,n,k} \cdot \rho, \sigma_{IIa,n,k} \cdot \pi) $ 
where the permutation
$\sigma_{IIa,n,k}: [n] \rightarrow [ n ] $ is given by
  $$\sigma_{IIa,n,k}(i)
    =
      \begin{cases}
        n-i +1   & \text{for $i\in \{1,2 \}$}\\
i + n-2-k & \text{for $3\leq i\leq  k$ }\\
        i  -k & \text{for $k+1\leq i\leq n$ }
      \end{cases}
$$
i.e. $\sigma_{IIa, n, k}$ exchanges the intervals $\{3,4,\dots ,k\}$ 
and $\{k+1,k+2,\dots,n\}$ and maps the pair $(1, 2)$ to $(n, n-1)$.
As above, note that $\sigma_{IIa,n,k} \cdot \rho \in \Int_{1,2}(n) $  and that $ \sigma_{IIa,n,k} \cdot \pi $ in $\oneconn{\sigma_{IIa, n, k} \cdot\rho} $ and the pair $(\sigma_{IIa,n,k} \cdot \rho,\sigma_{IIa,n,k} \cdot \pi)$ is in $\oneconntype{n}{IIb}$ 
  (see Figure~\ref{fig:exinvolutiontypeII} for a specific example).
  
  In the case when $ (\rho,\pi) $ is type  \textit{IIb} we reverse the above process. Definition~\ref{lemtwoinnerodd} gives that there exists $k\in [n-2]$ such that the padding interval of the unique 1-connected block is $\{ k+1, k+2,\dots, n-1\}$.
  In this case we define 
$\Phi(\rho, \pi) = (\sigma_{IIb,n,k} \cdot \rho, \sigma_{IIb,n,k} \cdot \pi) $ 
where the permutation
$\sigma_{IIb,n,k}: [n] \rightarrow [ n ] $ is given by
\[
  \sigma_{IIb,n,k}(i)
    =
      \begin{cases}
        i + n-k & \text{for $1\leq i\leq  k$ }\\
        i-k+2 & \text{for $ k+1\leq i\leq n-2$ }\\
        n+1-i   & \text{for $i\in \{ n-1, n \} $}\\
\end{cases}
\]

 Note that $\sigma_{IIa,n,k} =\sigma_{IIb,n,n-k}^{-1} $, and hence
  $ \Phi$ is an involution on $\oneconntype{n}{II}$.

  \begin{figure}[h]
    \begin{subfigure}[c]{0.3\textwidth}
      \centering
      \begin{tikzpicture}[scale=0.04]
        \draw (2,0)--(2,10.5);
        \draw (8,0)--(8,7.5);
        \draw[style=aline] (14,0)--(14,7.5);
        \draw[style=aline] (20,0)--(20,4.5);
        \draw[style=aline] (26,0)--(26,4.5);
        \draw[style=aline] (32,0)--(32,4.5);
        \draw[style=aline] (38,0)--(38,4.5);
        \draw[style=aline] (44,0)--(44,4.5);
        \draw[style=aline] (50,0)--(50,4.5);
        \draw[style=aline] (56,0)--(56,7.5);
        \draw[style=aline] (14,7.5)--(56,7.5);
        \draw[style=aline] (20,4.5)--(26,4.5);
        \draw[style=aline] (32,4.5)--(50,4.5);
        \draw[style=aline] (8,7.5)--(14,7.5);
        
        \draw[style=bline] (62,0)--(62,10.5);
        \draw[style=bline] (68,0)--(68,4.5);
        \draw[style=bline] (74,0)--(74,4.5);
        \draw[style=bline] (80,0)--(80,4.5);
        \draw[style=bline] (86,0)--(86,4.5);
        \draw[style=bline] (92,0)--(92,4.5);
        \draw[style=bline] (98,0)--(98,4.5);
        \draw[style=bline] (104,0)--(104,4.5);
        \draw[style=bline] (110,0)--(110,4.5);
        \draw[style=bline] (116,0)--(116,10.5);
        \draw[style=bline] (68,4.5)--(74,4.5);
        \draw[style=bline] (80,4.5)--(86,4.5);
        \draw[style=bline] (92,4.5)--(110,4.5);
        \draw[style=bline] (62,10.5)--(116,10.5);
        
        \draw (2,10.5)--(62,10.5);
        
        \draw[thick] (2,-9)--(2,-4.5);
        \draw[thick] (8,-9)--(8,-4.5);
        \draw (14,-9)--(14,-4.5);
        \draw (20,-9)--(20,-4.5);
        \draw (26,-9)--(26,-4.5);
        \draw (32,-9)--(32,-4.5);
        \draw (38,-9)--(38,-4.5);
        \draw (44,-9)--(44,-4.5);
        \draw (50,-9)--(50,-4.5);
        \draw (56,-9)--(56,-4.5);
        \draw (62,-9)--(62,-4.5);
        \draw (68,-9)--(68,-4.5);
        \draw (74,-9)--(74,-4.5);
        \draw (80,-9)--(80,-4.5);
        \draw (86,-9)--(86,-4.5);
        \draw (92,-9)--(92,-4.5);
        \draw (98,-9)--(98,-4.5);
        \draw (104,-9)--(104,-4.5);
        \draw (110,-9)--(110,-4.5);
        \draw (116,-9)--(116,-4.5);
        \draw[thick] (2,-4.5)--(8,-4.5);
        \draw (14,-4.5)--(20,-4.5);
        \draw (26,-4.5)--(32,-4.5);
\draw (50,-4.5)--(56,-4.5);
        \draw (62,-4.5)--(68,-4.5);
        \draw (74,-4.5)--(80,-4.5);
        \draw (86,-4.5)--(92,-4.5);
        \draw (98,-4.5)--(104,-4.5);
        \draw (110,-4.5)--(116,-4.5);
        
\end{tikzpicture}
      
    \end{subfigure}
    $\longleftrightarrow$
    \begin{subfigure}[c]{0.3\textwidth}
      \centering
      \begin{tikzpicture}[scale=0.04]
        \draw[style=bline] (2,0)--(2,10.5);
        \draw[style=bline] (8,0)--(8,4.5);
        \draw[style=bline] (14,0)--(14,4.5);
        \draw[style=bline] (20,0)--(20,4.5);
        \draw[style=bline] (26,0)--(26,4.5);
        \draw[style=bline] (32,0)--(32,4.5);
        \draw[style=bline] (38,0)--(38,4.5);
        \draw[style=bline] (44,0)--(44,4.5);
        \draw[style=bline] (50,0)--(50,4.5);
        \draw[style=bline] (56,0)--(56,10.5);
        \draw[style=bline] (2,10.5)--(56,10.5);
        
        \draw[style=bline] (8,4.5)--(14,4.5);
        \draw[style=bline] (20,4.5)--(26,4.5);
        \draw[style=bline] (32,4.5)--(50,4.5);

        \draw (56,10.5)--(116,10.5);
        
        \draw[style=aline] (62,0)--(62,7.5);
        \draw[style=aline] (68,0)--(68,4.5);
        \draw[style=aline] (74,0)--(74,4.5);
        \draw[style=aline] (80,0)--(80,4.5);
        \draw[style=aline] (86,0)--(86,4.5);
        \draw[style=aline] (92,0)--(92,4.5);
        \draw[style=aline] (98,0)--(98,4.5);
        \draw[style=aline] (104,0)--(104,7.5);
        \draw[style=aline] (62,7.5)--(104,7.5);
        
        \draw[style=aline] (68,4.5)--(74,4.5);
        \draw[style=aline] (80,4.5)--(98,4.5);
        
        \draw (104,7.5)--(110,7.5);

        \draw (110,0)--(110,7.5);
        \draw (116,0)--(116,10.5);

        \draw (2,-9)--(2,-4.5);
        \draw (8,-9)--(8,-4.5);
        \draw (14,-9)--(14,-4.5);
        \draw (20,-9)--(20,-4.5);
        \draw (26,-9)--(26,-4.5);
        \draw (32,-9)--(32,-4.5);
        \draw (38,-9)--(38,-4.5);
        \draw (44,-9)--(44,-4.5);
        \draw (50,-9)--(50,-4.5);
        \draw (56,-9)--(56,-4.5);
        \draw (62,-9)--(62,-4.5);
        \draw (68,-9)--(68,-4.5);
        \draw (74,-9)--(74,-4.5);
        \draw (80,-9)--(80,-4.5);
        \draw (86,-9)--(86,-4.5);
        \draw (92,-9)--(92,-4.5);
        \draw (98,-9)--(98,-4.5);
        \draw (104,-9)--(104,-4.5);
        \draw[thick] (110,-9)--(110,-4.5);
        \draw[thick] (116,-9)--(116,-4.5);
        
        \draw (2,-4.5)--(8,-4.5);
        \draw (14,-4.5)--(20,-4.5);
        \draw (26,-4.5)--(32,-4.5);
        \draw (38,-4.5)--(44,-4.5);
        \draw (50,-4.5)--(56,-4.5);
        \draw (62,-4.5)--(68,-4.5);
        \draw (74,-4.5)--(80,-4.5);
\draw (98,-4.5)--(104,-4.5);
        \draw[thick] (110,-4.5)--(116,-4.5);

      \end{tikzpicture}
      
    \end{subfigure}
    \vskip1ex
    \caption{An example on $\oneconntype{n}{II}$.} \label{fig:exinvolutiontypeII}
  \end{figure}

 \item type \textit{IIIa} is mapped to \textit{type IIIb} and vice versa, see
  Figure~\ref{fig:involutiontypeIII}.
  \begin{figure}[h]
    
    \begin{subfigure}[c]{0.35\textwidth}
      \centering
      \begin{tikzpicture}[scale=0.5]
        \draw [white] (0,3.0) circle (0);
        
        \path [afill,thick] (2.0,0) rectangle (7.2,1.0);
          \draw [thick] (0.4,0)--(0.4,1.6);
          \draw [thick] (-0.5,1.6)--(7.5,1.6);
\draw [thick] (1.3,0)--(1.3,1.0)--(2.0,1.0);
        \draw [dashed] (-1.8,1.6)--(10.3,1.6);
        
        \node at (1.2,-0.5){\footnotesize{k \ k+1}};
        \node at (7.2,-0.5){\footnotesize{l}};
      \end{tikzpicture} 
    \end{subfigure}
    $\qquad\quad\longleftrightarrow\qquad$
    \begin{subfigure}[c]{0.35\textwidth}
      \centering
      \begin{tikzpicture}[scale=0.5]
        \draw [white] (0,3.0) circle (0);
        \path [afill,thick] (1.1,0) rectangle (6.3,1.0);
        \draw [thick] (7.9,1.6)--(7.9,0);
        \draw [thick] (0.7,1.6)--(8.5,1.6);
\draw [thick] (7.2,0)--(7.2,1.0)--(6.3,1.0);
        \draw [dashed] (-1.8,1.6)--(10.3,1.6);
        
        \node at (7.6,-0.5){\footnotesize{l-1\ l}};
         \node at (1.1,-0.5){\footnotesize{k}};
      \end{tikzpicture} 
    \end{subfigure}
    
    \caption{Involution on $\oneconntype{n}{III}$.} \label{fig:involutiontypeIII}
  \end{figure}
\end{enumerate}

 Suppose first that $(\rho, \pi)\in\oneconntype{n}{IIIa}$. By Definition~\ref{lemtwoinnerodd}, let $k<l$ be such that the leftmost inner 1-connected block starts at $k+1$, ends at $l $, and $\{k, k+1\}$ be a brace in $\rho$. Then we define $\Phi((\rho, \pi)) = (\rho^\prime, \pi^\prime) $, where $\rho^\prime =\sigma_{IIIa, n, k, l} \cdot \rho $, and $ \pi^\prime = \sigma_{IIIa, n, k, l} \cdot \pi) $ with the permutation $\sigma_{IIIa, n, k, l}:[n] \rightarrow [n] $  given by
\[
  \sigma_{IIIa,n,k,l}(i)
  =
    \begin{cases}
      k+l-i  & \text{for $i\in \{k, k+1\}$}\\
i -2 & \text{for $k+2\leq i\leq  l$ }\\
      i & \text{for $i< k$ or $i > l$ .}
    \end{cases}
\]

From the definition of type \textit{IIIa}, we have that $\{1,2, \dots, k-1\}$, $\{ k, k+1\}$ and $\{  l+1, \dots, n \}$ are unions of blocks of $ \rho$, therefore 
$ \rho^\prime \in \Int_{1, 2}(n) $.
Also, note that, by construction, $(\rho^\prime,\pi^\prime)$ is of the type \textit{IIIb}.

  Finally, if $ (\rho, \pi)\in\oneconntype{n}{IIIb}$ then by Definition~\ref{lemtwoinnerodd} there are $k<l $ such that the leftmost inner 1-connected block of $ \pi$ starts at $k$, ends at $ l-1$ and $ \{l-1, l\}$ is a brace in $\rho$.  In this case we define $\Phi(\rho, \pi) =(\rho^\prime, \pi^\prime)$ with $\rho^\prime = \sigma_{IIIb, n, k, l} \cdot \rho$ and $\pi^\prime = \sigma_{IIIb, n, k, l} \cdot \pi$ where $ \sigma_{IIIb, n,k, l}:[n]\rightarrow [n]$ is the permutation given by
 \[  \sigma_{IIIb,n,k,l}(i)
 =
    \begin{cases}
      i+2  & \text{for $k\leq i\leq l-2$}\\
      k+l-i    & \text{for $i\in\{l-1, l\}$}\\       
i & \text{for $i < k$ or $i\geq l+1$. }
    \end{cases}
  \]
 Note that, as above $ \rho^\prime \in \Int_{1, 2}(n) $ and that $(\rho^\prime,\pi^\prime)$ of the type \textit{IIIa}.
 Moreover, we have that $\sigma_{IIIa,n,k,l} =\sigma_{IIIb,n,k,l}^{-1}$ so $\Phi$ is an involution also when $\pi$ is type \textit{III}. (See Figure~\ref{fig:exinvolutiontypeIII} below for a specific examples.)

\begin{figure}[h]
  \begin{subfigure}[c]{0.3\textwidth}
    \centering
    \begin{tikzpicture}[scale=0.04]
      \draw (2,0)--(2,10.5);
      \draw (8,0)--(8,4.5);
      \draw (14,0)--(14,4.5);
      \draw (20,0)--(20,10.5);
      \draw (26,0)--(26,10.5);
      \draw (32,0)--(32,7.5);
      
      \draw[style=aline] (38,0)--(38,7.5);
      \draw[style=aline] (44,0)--(44,4.5);
      \draw[style=aline] (50,0)--(50,4.5);
      \draw[style=aline] (56,0)--(56,4.5);
      \draw[style=aline] (62,0)--(62,4.5);
      \draw[style=aline] (68,0)--(68,4.5);
      \draw[style=aline] (74,0)--(74,4.5);
      \draw[style=aline] (80,0)--(80,7.5);
      \draw[style=aline] (44,4.5)--(50,4.5);
      \draw[style=aline] (56,4.5)--(74,4.5);
      \draw[style=aline] (38,7.5)--(80,7.5);

      \draw (86,0)--(86,10.5);
      \draw (92,0)--(92,4.5);
      \draw (98,0)--(98,4.5);
      \draw (104,0)--(104,4.5);
      \draw (110,0)--(110,4.5);
      \draw (116,0)--(116,10.5);
      \draw (8,4.5)--(14,4.5);
      \draw (32,7.5)--(38,7.5);
      \draw (92,4.5)--(110,4.5);
      \draw (2,10.5)--(116,10.5);

      \draw (2,-9)--(2,-4.5);
      \draw (8,-9)--(8,-4.5);
      \draw (14,-9)--(14,-4.5);
      \draw (20,-9)--(20,-4.5);
      \draw[thick] (26,-9)--(26,-4.5);
      \draw[thick] (32,-9)--(32,-4.5);
      \draw (38,-9)--(38,-4.5);
      \draw (44,-9)--(44,-4.5);
      \draw (50,-9)--(50,-4.5);
      \draw (56,-9)--(56,-4.5);
      \draw (62,-9)--(62,-4.5);
      \draw (68,-9)--(68,-4.5);
      \draw (74,-9)--(74,-4.5);
      \draw (80,-9)--(80,-4.5);
      \draw (86,-9)--(86,-4.5);
      \draw (92,-9)--(92,-4.5);
      \draw (98,-9)--(98,-4.5);
      \draw (104,-9)--(104,-4.5);
      \draw (110,-9)--(110,-4.5);
      \draw (116,-9)--(116,-4.5);
      \draw (2,-4.5)--(8,-4.5);
      \draw (14,-4.5)--(20,-4.5);
      \draw[thick] (26,-4.5)--(32,-4.5);
      \draw (38,-4.5)--(44,-4.5);
      \draw (50,-4.5)--(56,-4.5);
      \draw (62,-4.5)--(68,-4.5);
      \draw (74,-4.5)--(80,-4.5);
      \draw (86,-4.5)--(92,-4.5);
      \draw (98,-4.5)--(104,-4.5);
      \draw (110,-4.5)--(116,-4.5);
      
\end{tikzpicture}       
  \end{subfigure}
  $\longleftrightarrow$
  \begin{subfigure}[c]{0.3\textwidth}
    \centering
    \begin{tikzpicture}[scale=0.04]
      \draw (2,0)--(2,10.5);
      \draw (8,0)--(8,4.5);
      \draw (14,0)--(14,4.5);
      \draw (20,0)--(20,10.5);
      
      \draw[style=aline] (26,0)--(26,7.5);
      \draw[style=aline] (32,0)--(32,4.5);
      \draw[style=aline] (38,0)--(38,4.5);
      \draw[style=aline] (44,0)--(44,4.5);
      \draw[style=aline] (50,0)--(50,4.5);
      \draw[style=aline] (56,0)--(56,4.5);
      \draw[style=aline] (62,0)--(62,4.5);
      \draw[style=aline] (68,0)--(68,7.5);
      \draw[style=aline] (26,7.5)--(68,7.5);
      \draw[style=aline] (32,4.5)--(38,4.5);
      \draw[style=aline] (44,4.5)--(62,4.5);
      
      \draw (68,7.5)--(74,7.5);
      
      \draw (74,0)--(74,7.5);
      \draw (80,0)--(80,10.5);
      \draw (86,0)--(86,10.5);
      \draw (92,0)--(92,4.5);
      \draw (98,0)--(98,4.5);
      \draw (104,0)--(104,4.5);
      \draw (110,0)--(110,4.5);
      \draw (116,0)--(116,10.5);
      \draw (8,4.5)--(14,4.5);
      \draw (92,4.5)--(110,4.5);
      \draw (2,10.5)--(116,10.5);

      \draw (2,-9)--(2,-4.5);
      \draw (8,-9)--(8,-4.5);
      \draw (14,-9)--(14,-4.5);
      \draw (20,-9)--(20,-4.5);
      \draw (26,-9)--(26,-4.5);
      \draw (32,-9)--(32,-4.5);
      \draw (38,-9)--(38,-4.5);
      \draw (44,-9)--(44,-4.5);
      \draw (50,-9)--(50,-4.5);
      \draw (56,-9)--(56,-4.5);
      \draw (62,-9)--(62,-4.5);
      \draw (68,-9)--(68,-4.5);
      \draw[thick] (74,-9)--(74,-4.5);
      \draw[thick] (80,-9)--(80,-4.5);
      \draw (86,-9)--(86,-4.5);
      \draw (92,-9)--(92,-4.5);
      \draw (98,-9)--(98,-4.5);
      \draw (104,-9)--(104,-4.5);
      \draw (110,-9)--(110,-4.5);
      \draw (116,-9)--(116,-4.5);
      \draw (2,-4.5)--(8,-4.5);
      \draw (14,-4.5)--(20,-4.5);
      \draw (26,-4.5)--(32,-4.5);
      \draw (38,-4.5)--(44,-4.5);
      \draw (50,-4.5)--(56,-4.5);
      \draw (62,-4.5)--(68,-4.5);
      \draw[thick] (74,-4.5)--(80,-4.5);
      \draw (86,-4.5)--(92,-4.5);
      \draw (98,-4.5)--(104,-4.5);
      \draw (110,-4.5)--(116,-4.5);
      
\end{tikzpicture}       
  \end{subfigure}
  \vskip5ex
  
  \begin{subfigure}[c]{0.3\textwidth}
    \centering
    \begin{tikzpicture}[scale=0.04]
      \draw (2,0)--(2,10.5);
      \draw (8,0)--(8,7.5);
      
      \draw[style=aline] (14,0)--(14,7.5);
      \draw[style=aline] (20,0)--(20,4.5);
      \draw[style=aline] (26,0)--(26,4.5);
      \draw[style=aline] (32,0)--(32,4.5);
      \draw[style=aline] (38,0)--(38,4.5);
      \draw[style=aline] (44,0)--(44,4.5);
      \draw[style=aline] (50,0)--(50,4.5);
      \draw[style=aline] (56,0)--(56,7.5);
      \draw[style=aline] (14,7.5)--(56,7.5);
      \draw[style=aline] (20,4.5)--(26,4.5);
      \draw[style=aline] (32,4.5)--(50,4.5);
      
      \draw (62,0)--(62,10.5);
      \draw (68,0)--(68,10.5);
      \draw (74,0)--(74,7.5);
      \draw (80,0)--(80,4.5);
      \draw (86,0)--(86,4.5);
      \draw (92,0)--(92,7.5);
      \draw (98,0)--(98,7.5);
      \draw (104,0)--(104,4.5);
      \draw (110,0)--(110,4.5);
      \draw (116,0)--(116,10.5);
      
      \draw (8,7.5)--(14,7.5);
      \draw (2,10.5)--(62,10.5);
      
      \draw (80,4.5)--(86,4.5);
      \draw (74,7.5)--(98,7.5);
      \draw (104,4.5)--(110,4.5);
      \draw (68,10.5)--(116,10.5);
      
      \draw[thick] (2,-9)--(2,-4.5);
      \draw[thick] (8,-9)--(8,-4.5);
      \draw (14,-9)--(14,-4.5);
      \draw (20,-9)--(20,-4.5);
      \draw (26,-9)--(26,-4.5);
      \draw (32,-9)--(32,-4.5);
      \draw (38,-9)--(38,-4.5);
      \draw (44,-9)--(44,-4.5);
      \draw (50,-9)--(50,-4.5);
      \draw (56,-9)--(56,-4.5);
      \draw (62,-9)--(62,-4.5);
      \draw (68,-9)--(68,-4.5);
      \draw (74,-9)--(74,-4.5);
      \draw (80,-9)--(80,-4.5);
      \draw (86,-9)--(86,-4.5);
      \draw (92,-9)--(92,-4.5);
      \draw (98,-9)--(98,-4.5);
      \draw (104,-9)--(104,-4.5);
      \draw (110,-9)--(110,-4.5);
      \draw (116,-9)--(116,-4.5);
      \draw[thick] (2,-4.5)--(8,-4.5);
      \draw (14,-4.5)--(20,-4.5);
      \draw (26,-4.5)--(32,-4.5);
      \draw (38,-4.5)--(44,-4.5);
      \draw (50,-4.5)--(56,-4.5);
      \draw (62,-4.5)--(68,-4.5);
      \draw (74,-4.5)--(80,-4.5);
      \draw (86,-4.5)--(92,-4.5);
      \draw (98,-4.5)--(104,-4.5);
      \draw (110,-4.5)--(116,-4.5);
      
\end{tikzpicture}
    
  \end{subfigure}
  $\longleftrightarrow$
  \begin{subfigure}[c]{0.3\textwidth}
    \centering
    \begin{tikzpicture}[scale=0.04]
      \draw[style=aline] (2,0)--(2,7.5);
      \draw[style=aline] (8,0)--(8,4.5);
      \draw[style=aline] (14,0)--(14,4.5);
      \draw[style=aline] (20,0)--(20,4.5);
      \draw[style=aline] (26,0)--(26,4.5);
      \draw[style=aline] (32,0)--(32,4.5);
      \draw[style=aline] (38,0)--(38,4.5);
      \draw[style=aline] (44,0)--(44,7.5);
      \draw[style=aline] (2,7.5)--(44,7.5);
      \draw[style=aline] (8,4.5)--(14,4.5);
      \draw[style=aline] (20,4.5)--(38,4.5);
      
      \draw (50,0)--(50,7.5);
      \draw (56,0)--(56,4.5);
      \draw (62,0)--(62,4.5);
      \draw (68,0)--(68,10.5);
      \draw (74,0)--(74,7.5);
      \draw (80,0)--(80,4.5);
      \draw (86,0)--(86,4.5);
      \draw (92,0)--(92,7.5);
      \draw (98,0)--(98,7.5);
      \draw (104,0)--(104,4.5);
      \draw (110,0)--(110,4.5);
      \draw (116,0)--(116,10.5);
      
      \draw (44,7.5)--(50,7.5);
      \draw (56,4.5)--(62,4.5);
      \draw (80,4.5)--(86,4.5);
      \draw (74,7.5)--(98,7.5);
      \draw (104,4.5)--(110,4.5);
      \draw (68,10.5)--(116,10.5);

      \draw (2,-9)--(2,-4.5);
      \draw (8,-9)--(8,-4.5);
      \draw (14,-9)--(14,-4.5);
      \draw (20,-9)--(20,-4.5);
      \draw (26,-9)--(26,-4.5);
      \draw (32,-9)--(32,-4.5);
      \draw (38,-9)--(38,-4.5);
      \draw (44,-9)--(44,-4.5);
      \draw[thick] (50,-9)--(50,-4.5);
      \draw[thick] (56,-9)--(56,-4.5);
      \draw (62,-9)--(62,-4.5);
      \draw (68,-9)--(68,-4.5);
      \draw (74,-9)--(74,-4.5);
      \draw (80,-9)--(80,-4.5);
      \draw (86,-9)--(86,-4.5);
      \draw (92,-9)--(92,-4.5);
      \draw (98,-9)--(98,-4.5);
      \draw (104,-9)--(104,-4.5);
      \draw (110,-9)--(110,-4.5);
      \draw (116,-9)--(116,-4.5);
      \draw (2,-4.5)--(8,-4.5);
      \draw (14,-4.5)--(20,-4.5);
      \draw (26,-4.5)--(32,-4.5);
      \draw (38,-4.5)--(44,-4.5);
      \draw[thick] (50,-4.5)--(56,-4.5);
      \draw (62,-4.5)--(68,-4.5);
      \draw (74,-4.5)--(80,-4.5);
      \draw (86,-4.5)--(92,-4.5);
      \draw (98,-4.5)--(104,-4.5);
      \draw (110,-4.5)--(116,-4.5);

    \end{tikzpicture}
    
  \end{subfigure}

  \caption{Examples of the involution on $\oneconntype{n}{III}$.} \label{fig:exinvolutiontypeIII}
\end{figure}
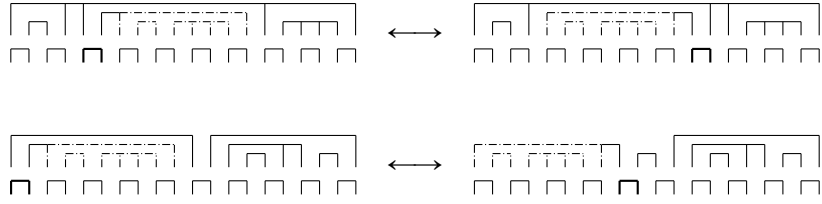

\begin{Rem}[Why types I and II are necessary]
\label{rem:why-types-I-II}
It may seem tempting to dispense with types~I and~II and to apply the
type~III operation to the leftmost inner 1-connected block of every
1-connected partition. The following example shows that the resulting map
would not be injective. Put
\[
\begin{aligned}
 \pi_L&=\bigl\{\{1,5,6\},\{2,3,4\}\bigr\},
 &\rho_L&=\bigl\{\{1,2\},\{3\},\{4\},\{5\},\{6\}\bigr\},\\
 \pi_R&=\bigl\{\{1,2,6\},\{3,4,5\}\bigr\},
 &\rho_R&=\bigl\{\{1\},\{2\},\{3\},\{4\},\{5,6\}\bigr\}.
\end{aligned}
\]
In $(\rho_L,\pi_L)$, the block $\{2,3,4\}$ is the unique 1-connected
block and has pivot brace $\{1,2\}$; thus $\pi_L$ is of type~IIa.
Similarly, $\{3,4,5\}$ is the unique 1-connected block of $\pi_R$ and
has pivot brace $\{5,6\}$, so $\pi_R$ is of type~IIb.

Suppose nevertheless that we apply the type~IIIa rule to the first pair and
the type~IIIb rule to the second one. Then
\[
 \bigl(\sigma_{IIIa,6,1,4}\mathbin{\cdot}\rho_L,
       \sigma_{IIIa,6,1,4}\mathbin{\cdot}\pi_L\bigr)
 =(\rho_0,\pi_0)
 =\bigl(\sigma_{IIIb,6,3,6}\mathbin{\cdot}\rho_R,
       \sigma_{IIIb,6,3,6}\mathbin{\cdot}\pi_R\bigr),
\]
where
\[
 \rho_0=\bigl\{\{1\},\{2\},\{3,4\},\{5\},\{6\}\bigr\},
 \qquad
 \pi_0=\bigl\{\{1,2,3\},\{4,5,6\}\bigr\}.
\]

Both blocks of $\pi_0$ are 1-connected and have the common pivot
brace $\{3,4\}$; hence $\pi_0$ is precisely a type~I partition. Thus the
naive extension of the type~III rule sends two distinct type~II boundary
configurations to the same type~I configuration. In particular, it is not
injective and therefore cannot be an involution suitable for the cancellation
argument we make later. Indeed, applying the same rule to the leftmost 1-connected block of
$\pi_0$ gives
\[
 \widetilde\Phi^2(\rho_R,\pi_R)
 =\widetilde\Phi(\rho_0,\pi_0)
 =(\rho_L,\pi_L)\ne(\rho_R,\pi_R).
\]

This explains the separate treatment of types~I and~II. The type~II rule
pairs the two boundary configurations directly,
\[
 \Phi(\rho_L,\pi_L)=(\rho_R,\pi_R),
 \qquad
 \Phi(\rho_R,\pi_R)=(\rho_L,\pi_L),
\]
whereas configurations with two 1-connected blocks sharing a pivot brace are
kept in type~I and treated by the left--right flip prescribed there.
\end{Rem}
  An immediate consequence of the construction of $\Phi$ is the following remark.
\begin{Rem}\label{rem:phi}
    Suppose that $ (\rho, \pi)\in\oneconn{n}$ and  $\sigma:[n] \rightarrow [n]$ the permutation such that $ \Phi(\rho, \pi) = (\sigma\cdot \rho, \sigma\cdot \pi)$. 
    \begin{itemize}[leftmargin=2em]
        \item[1.] If $ \{i, i+1\} $ is a brace of $\rho$ that is not the pivot of $\pi$, then $ \sigma(i+1) = \sigma(i) +1$; if $\{i, i+1\}$ is the pivot of $\pi$, then $\sigma(i+1) =\sigma(i) -1 $; in particular, $ \sigma $ inverses the order of exaclty one brace of $\rho$.
        \item[2.] If  $\{i_1, i_2, \dots, i_q\}$ is a block of $ \pi $ such that $i_1< i_2< \cdots < i_q$ then there are three possibilties: $\sigma(i_1) < \sigma(i_2) < \cdots < \sigma (i_q) $ or $ \sigma(i_2) < \cdots < \sigma(i_q)< \sigma(i_1)$ or $\sigma(i_q) < \sigma(i_1) \cdots < \sigma(i_{q-1})$, so $ \sigma$  on the blocks of $\pi$ is either identity or a circular permutation.
    \end{itemize}
\end{Rem}

	\section{Free Cumulants of Inhomogenous Quadratic Forms}	
\label{sec:cumulants}
\subsection{General case}
  The framework of this section is a non-commutative probability space
  $(\A,\tau)$ and $ X_1, X_2, \dots, X_n $ a family of self-adjoint elements in
  $ \A$. The main result, Proposition~\ref{Prop:45} is a formula that gives a
  combinatorial description of free cumulants of the element
  \begin{equation}
    \label{eq:T}
     T = \sum_{i, j=1}^n a_{i,j} X_iX_j + \sum_{i=1}^n b_i X_i,
  \end{equation}
where $ b_i \in \mathbb{R}$ for each $i $ and  $ A =\big[ a_{i, j}\big]_{i, j=1}^n$ is a self-adjoint skew-symmetric matrix with complex entries.

\begin{defi}Let $n $ be a positive integer.
 \begin{enumerate}[1.,leftmargin=2em]
 \item An $n$-coloring of the set $[m]$ is a function $c:[m]\rightarrow [n]$.
 \item Given a coloring $c:[m]\to[n]$ we consider its \emph{cyclic segmentation},
  which is the  partition  $\sigma\in NC(m)$
  arising from the equivalence relation
  $i\sim_\sigma (i+1) \mod m$ $\iff{}$ $c(i)=c((i+1) \mod m)$.
  We denote this partition by $\seg(c)$.
 \item For $\rho\in \Int_{1, 2}(m)$ and $ c: [m]\rightarrow [n]$ we define
 \[ \Prodab (\rho, c)= \prod_{\substack{\{i<j\}\in \rho }}  a_{c(i), c(j)} \cdot \prod_{\{k\} \in \rho} b_{c(k)}
 \]
 \item For $\pi \in NC(m)$ and $c$ an $n$-coloring of $[m]$, we define
 \[
 \kappa(\pi, c)= \kappa_{\pi}\bigl( X_{c(1)}, X_{c(2)}, \dots, X_{c(m)}   \bigr).
 \]
 \end{enumerate}
\end{defi}
    With these notations
    \[ \kappa_{m}(T) =
    \sum_{k=0}^m 
    \sum_{c: [m+k]\to [n]} 
    \sum_{\substack{\rho\in\Int_{1, 2}(m+k) \\ \rho =\{ B_1, B_2, \dots, B_m\} }}
    \Prodab(\rho, c)
    \,
    \kappa_m \Bigl( \prod_{l\in B_1} X_{c(l)}, \prod_{l\in B_2} X_{c(l)}, \dots, \prod_{l\in B_m} X_{c(l)}\Bigr)
    \]
	where the products are in the natural order. Using the formula for cumulants with products as entries from Theorem~\ref{thm:krawczyk}, with $\rho$ denoting the product degree partition the formula above becomes
   \begin{align}\label{kappa:1}
   \kappa_m(T) &= 
 \sum_{k=0}^m 
 \sum_{c: [m+k]\to [n]}
\sum_{{\substack{\rho\in\Int_{1, 2}(m+k) \\ |\rho| =m  }}} 
 \Prodab(\rho, c) \sum_{\substack{\pi\in\uComp\rho }}
 \kappa(\pi, c). \\ 
                &=
\label{eq:kappam(T)=C1uC2}                  
 \sum_{k=0}^m 
 \sum_{c: [m+k]\to [n]}
\sum_{{\substack{(\rho,\pi)\in\oneconn{m+k}\sqcup\twoconn{m+k} \\ |\rho| =m  }}} 
 \Prodab(\rho, c) \,
 \kappa(\pi, c).
   \end{align}
Next, we show that the contributions of the configurations $(\pi,\rho)\in\oneconn{m+k}$ cancel, for every $m,k$ integers.
\begin{theo}\label{cancellation}
  Suppose that the matrix $A$ is skew-symmetric
  and $ \tau $ is tracial
  (the latter is true in particular when $X_1,X_2,\ldots,X_n$ form a free
  family),  then the cumulants of \eqref{eq:T} are given by
  \begin{align} \label{eqn:thm42}
    \kappa_m(T)
    &= 
 \sum_{k=0}^m 
 \sum_{c: [m+k]\to [n]}
\sum_{{\substack{\rho\in\Int_{1, 2}(m+k) \\ |\rho| =m  }}} 
 \Prodab(\rho, c) \sum_{\substack{\pi\geq \bar \rho }}
    \kappa(\pi, c)\\
\label{eq:kappam(T)=C2}
    &=
 \sum_{k=0}^m 
 \sum_{c: [m+k]\to [n]}
\sum_{{\substack{(\rho,\pi)\in\twoconn{m+k} \\ |\rho| =m  }}} 
 \Prodab(\rho, c) \,
 \kappa(\pi, c).
    \end{align}
\end{theo}
\begin{proof}
  Comparing \eqref{eq:kappam(T)=C1uC2}
  with \eqref{eq:kappam(T)=C2}
  we have to show that the contributions of $\oneconn{}$ cancel.
  For this
 it suffices to show that for any $(\rho, \pi) \in \oneconn{m+k}$ and for any $n$-coloring $ c $   of $ [m+k]$ we have that  
\[
     \Prodab(\rho, c)\, \kappa(\pi, c) = - \Prodab(\sigma \cdot \rho, c\circ \sigma^{-1})\, \kappa(\sigma \cdot \pi, c \circ \sigma^{-1} ),
\]

where $ \sigma $ is the permutation on $[m+k]$ which implements the involution $\Phi((\rho, \pi)) = (\sigma \cdot \rho, \sigma \cdot \pi)$ from Section~\ref{sec:involution}.

 Assume that $ (\rho, c) \in \oneconn{m+k} $. Since the matrix $ \big[ a_{i, j} \big]_{i, j =1}^n$ is skew-symmetric, part 1 of  Remark~\ref{rem:phi} gives that 
 $\Prodab(\rho, c) = -\Prodab(\sigma\cdot \rho, c\circ \sigma^{-1})$. 
 Furthermore, the traciality of $ \tau $ and  part 2 of Remark~\ref{rem:phi} gives that 
 $\kappa(\pi, c) = \kappa(\sigma \cdot \pi, c \circ \sigma^{-1}) $,
 and the conclusion follows.
\end{proof}

\begin{cor}
  Observe that
  configurations $(c,\rho)$ with braces $\{i,i+1\}\in\rho$
  such that $c(i)=c(i+1)$ give zero contribution to \eqref{eqn:thm42} because
  in this case $a_{c(i),c(i+1)}=0$.
  The condition $\seg(c)\leq \bar \rho$ excludes such configurations
  and thus we can write
  \begin{equation} \label{eqn:cor43}
    \kappa_m(T) = 
 \sum_{k=0}^m 
\sum_{{\substack{\rho\in\Int_{1, 2}(m+k) \\ |\rho| =m  }}}
\sum_{\substack{c: [m+k]\to [n]\\ \seg(c)\leq \bar \rho}}
 \Prodab(\rho, c) \sum_{\substack{\pi\geq \bar \rho }}
 \kappa(\pi, c).
\end{equation}
\end{cor}
\subsection{Case of free variables}
In this subsection we assume that $X_1,X_2,\ldots,X_n$ are free. In this case mixed free cumulants vanish, which puts additional
restrictions on the configurations appearing in \eqref{eqn:cor43}.
We note that the colorings in \eqref{eqn:cor43} have fixed segmentation and $\pi$ must be a so-called colored noncrossing partition.
\begin{defi}\label{def:admi}
     Fix a coloring $c:[m]\to[n]$. A partition $\pi\in NC(m)$ is called \emph{colored noncrossing partition}, denoted $\pi\in NC(n,c)$ if $c$ is constant on the blocks of $\pi$. That is of $i$ and $j$ are in the same block of $\pi$ then $c(i)=c(j)$.  We say that a block has the color of its elements.
     \end{defi}
\begin{prop}\label{Prop:45}
    For a fixed $\rho\in\Int_{1,2}$ the colorings $c:[m+k]\to[n]$ which contribute to the sum \eqref{eqn:cor45} are those satisfying $\seg(c)=\bar \rho$. Thus
    \begin{equation}
    \label{eqn:cor45}    \kappa_m(T)= 
 \sum_{k=0}^m 
\sum_{{\substack{\rho\in\Int_{1, 2}(m+k) \\ |\rho| =m  }}}
\sum_{\substack{c: [m+k]\to [n]\\ \seg(c)= \bar \rho}}
 \Prodab(\rho, c) \sum_{\substack{\pi\geq \bar \rho\\\pi \in NC(m+k,c) }}
 \kappa(\pi, c).
 \end{equation}

\end{prop}
\begin{proof}
  The freeness assumption implies  that $\pi$ is colored
  and since $\pi\geq \bar \rho$,
  the coloring must be constant on the blocks of
  $\bar \rho$ as well. Thus the colors can differ only on consecutive elements
  which form braces of $\rho$.
\end{proof}

\section{Operator model}

Let $(X_i)_{i\geq 1}$ be a sequence of free random variables with compound free Poisson distributions. Specifically, assume that for each $1\leq i \leq n$, the free cumulants of $X_i$ are given by
\begin{align*}
    \kappa_m(X_i) = \int_{\mathbb{R}} t^m \, d\sigma_i(t) \quad \text{for all } m \geq 1,
\end{align*}
where $\sigma_i$ is a positive compactly supported measure on $\mathbb{R}$.

For each $i \in \{1,2,\ldots,n\}$, let $H_i = L^2(\mathbb{R}, \sigma_i)$,
put $\xi_i=\mathbb{1}_{H_i}$  the constant function
and $\mu_i \in H_i$   the identity function $\mu_i(t)=t$.
Denote by $M_i=M_{\mu_i}$  the multiplication operator on $H_i$ defined by
\begin{align*}
    (M_i f)(t) = t f(t).
\end{align*}

Put $H = \bigoplus_{i=1}^n H_i$ and define the Fock space (without vacuum) as
$\mathcal{H} = \bigoplus_{k=1}^\infty H^{\otimes k}$.
We will tacitly
identify the subspaces $H_i\subseteq H\subseteq\mathcal{H}$ with their respective embeddings.
Denote by $P_i$  the orthogonal projection from $H$ onto $H_i$
and by $P_i^\perp = I - P_i$ 
the projection onto the orthogonal complement $H_i^\perp\subseteq H$.

We define the following operators acting on elementary tensors $v = v_1 \otimes v_2 \otimes \cdots \otimes v_k \in \mathcal{H}$. For $1 \leq k \leq n$, let:
\begin{align*}
    B_i(v) &= M_i(P_i v_1) \otimes v_2 \otimes \cdots \otimes v_k, \\
  A_i^*(v) &= \mu_i \otimes
v_1 \otimes v_2 \otimes \cdots \otimes v_k, \\
  A_i(v) &=\begin{cases}
      0 &\text{ if } k=1,\\
      \langle v_1, \mu_i \rangle\,
v_2 \otimes v_3 \otimes \cdots \otimes v_k &\text{ if } k\geq 2.
  \end{cases} 
\end{align*}
These operators are extended to $\mathcal{H}$ by linearity and continuity.

We define the operator $V$ on $\mathcal{H}$ by:
\begin{align*}
    V = \sum_{k=1}^n b_k B_k + \sum_{\substack{1\leq i,j\leq n \\ i\neq
  j}}a_{i,j} \left( B_i A_j + A_i^* B_j + A_i^* A_j \right).
\end{align*}

Finally, define the vector $\xi = \xi_1 \oplus \xi_2 \oplus \cdots \oplus \xi_n \in H \subset \mathcal{H}$.

\begin{theo}\label{twr:niskonczonapodzielnoscEL}
    For each $m \geq 1$, we have
    \[
        \kappa_m(T) = \langle V^m \xi, \xi \rangle.
    \]
\end{theo}
\begin{proof}
We show that the expansion of $\langle V^m\xi,\xi\rangle$ reproduces exactly
the formula \eqref{eqn:cor45}, by establishing a bijection between the non-zero terms in both expansions.
To this end we will expand the support of $V^m$ into 
 two alphabets:
\begin{align*}
  \mathbb{S}&=\bigcup_{j=1}^n\{B_j,A_j,A_j^*\}\\
  \mathbb{T}&=\bigcup_{1\leq i\neq j\leq n}\{B_i,B_iA_j,A_i^*B_j,A_i^* A_j\}\}
               .
\end{align*}

We denote by $\mathcal{W}_m$ the set of words of length $m$ over the alphabet
$\mathbb{T}$, i.e.
\[
  \mathcal{W}_m=\Big\{T_1T_2 \dotsm T_m \mid T_i\in\mathbb{T} \Big\}
  .
\] 

Observe that formally, as long as we do not simplify products, every $w\in\mathcal{W}_m$
has a unique factorization into $m$ elements from $\mathbb{T}$.

We shall establish a bijection between the set of words $w\in\mathcal{W}_m$ with $\langle w \xi,\xi\rangle\neq 0$
and the  tuples $(k,\rho,c,\pi)$ appearing in \eqref{eqn:cor45}.

Pick a word $w\in \mathcal{W}_m$ with non-zero contribution.n
It has a unique factorization $w=S_1S_2\dotsm S_{m+k}$ with $S_i\in
\mathbb{S}$ for every $1\leq i\leq m +k$, and so $k$ is uniquely determined.

The coloring $c:[m+k]\to[n]$ is given by the colors of the letters induced by
their indices, that is,
for $1\leq i\leq m+k$  put $c(i)=k$ if $S_i\in\{B_k,A_k,A_k^*\}$.

The Fibonacci partition $\rho$ can be recovered from the unique factorization
of $w$ into alphabet $\mathbb{T}$: each  $B_iA_j,A_i^*B_j$ and ${A_i^*A_j}$
corresponds to a brace and the remaining points are singletons.

It remains to show that  that $\bar \rho=\seg(c)$
and to construct $\pi\in NC(m+k,c)$ such that $\pi\geq\bar\rho$.

\noindent
\textbf{Claim 1:} $w$ starts and ends in $B_i$.
Indeed assume otherwise that $w=A_j^*w'$, then $\langle w\xi,\xi\rangle= \langle w'\xi,A_j\xi\rangle=0$. Similarly if $w$ does not end in $B_i$ then it ends in some $A_j$, but $A_j\xi=0$.  

\noindent
\textbf{Claim 2:} $w$ is a \emph{Motzkin word}, that is,
it contains the same number of creation and annihilation operators
and in each initial segment $S_1S_2\dotsm S_i$ of $w$  the number of creation
operators $A_j^*$ does not exceed the number of annihilation operators  $A_j$.

Let $r$ be the total number of creation and annihilation operators appearing
in $w$ with $k\in\{1,2,\ldots,n\}$.
Since creation operators increase the length of tensors by one and annihilation operators decrease it by one,
a necessary condition for $\langle w\xi,\xi\rangle\neq0$ is that $r$ is even
and there is the same number of creation and annihilation operators in $w$.
If the second condition is violated for a segment $S_1S_2\dotsm S_i$ then
\begin{align*}
  \langle w\xi,\xi\rangle
  &= \langle S_1S_2\dotsm S_{m+k}\xi,\xi\rangle\\
  &= \langle S_{i+1}S_{i+2}\dotsm S_{m+k}\xi,S_i^*\dotsm S_2^*S_1^*\xi\rangle
\end{align*}
and $S_i^*\dotsm S_2^*S_1^*\xi=0$.

\noindent
{\bf Claim 3:} $w$ corresponds to  an  irreducible noncrossing partition $\pi_w\in NC(m+k,c)$ without singletons.

Let $w=S_1S_2\dotsm S_{m+k}$ be a word appearing in the expansion of $V^m$, where $S_i\in\mathbb{S}$.
We proceed by the induction on the total number $r$  of creation and
annihilation operators $A_j$ and $A_j^*$ with $1\leq j \leq n$. 

\noindent
\textbf{Base case $r=0$.}
In this case $w$ consists only of operators $B_i$.
Since $B_i\xi = M_i\xi_i\in H_i$ and $H_i\perp H_j$ for $i\neq j$,
a nonzero contribution occurs if and only if all indices coincide.
In this case $w$ corresponds to the unique irreducible partition with one block.

\medskip

\noindent
\textbf{Induction step.}
Suppose that $w$ contains exactly  $r\ge2$ creation and annihilation operators
and assume that the assertion holds for all words with fewer than $r$ such
operators. 

Scan the word from left to right and let $q$ be the smallest index such that
$S_q=A_j^*$ for some $j$. Since $w$ is a Motzkin word there exists an
annihilation operator to the left of $S_q$. Let $p<q$ be the largest index $p$
such that $S_p=A_i$ for some $i$.
By the choice of $p$ and $q$, each operator between $p$ and $q$ is of the form $B_l$ for some $l$.
We claim that $i=j$ and all operators between $p$ and $q$ are be equal to $B_j$.
Indeed, 
 if $B_l$ with $l\neq j$ appears between $p$ and $q$, then $B_l$ acts on a tensor
whose first component lies in $H_j$, hence vanishes.
Therefore only $B_j$ may occur between $A_j^*$ and $A_i$.

Similarly if $i\neq j$, then $A_i$ acts on a tensor whose first component is
orthogonal to $H_i$, therefore
$\langle w\xi,\xi\rangle=0$.
Thus $i=j$.

Consequently, the subword
\[
A_j\,B_j^{\,q-p-1}\,A_j^*.
\]
gives rise to an interval block $\{p,\ldots,q\}\in\pi_w$ of size at least two, whose maximal element corresponds to
a creation operator and whose minimal element corresponds to an annihilation
operator. Observe that all elements of this block have the same color $j$.

After removing the above subword we are left with a shorter Motzkin word $w'$ with $r-2$ creation and
annihilation operators. 
By the induction hypothesis the Motzkin word $w'$ corresponds to a noncrossing
irreducible partition without singletons. Inserting the interval block
constructed
above produces a noncrossing partition with the same properties.

Thus $\pi_w$ is constructed. Observe that by the definition of the alphabet
$\mathbb{T}$ for each inner block $\{p,\ldots,q\}$ constructed above we have
$S_{p-1}\in \{B_l,A_l^*\}$ with $l\neq j$ similarly $S_{q+1}\in\{B_l, A_l\}$
with $l\neq j$. Moreover the outer block of $\pi_w$ contains the outer block of
$\bar \rho$, as from the construction, the outer block of $\pi$ contains all
$B_i$ at the beggining and end of $w$.
Thus $\pi\geq \bar \rho$.

Given $(k,\rho,c,\pi)$ it is clear how to define the corresponding $w\in \mathcal{W}_m$.

\noindent
\textbf{Evaluation of the blocks.}
Let $\{p,\ldots,q\}$ be the inner most block of $\pi_w$, we have
\[
\langle w\xi,\xi\rangle
=
\kappa_{q-p+1}(X_j)\,\langle w'\xi,\xi\rangle,
\]
hence recursively $\langle w\xi,\xi\rangle=\kappa(\pi,c).$
\medskip

In order to obtain the coefficient $\alpha(\pi,c)$ observe that every $w\in\mathcal{W}_m$ can be decomposed in a unique way as a product of of the alphabet $\mathbb{T}$. For example if  $w=w_1B_i A_j w_2$ then this $B_i A_j$ cannot be generated otherwise as from the building block $B_i A_j$, as otherwise $A_j$ is multiplied form the left by $A_j^*$. The same applies to all pairs of neighbouring elements of the form $\{B_iA_j,A_j^*B_i,A_i^* A_j\}$, after this identification we are left only with $B_k$. For example 
\[w=B_i B_i A_j A_j^* A_k B_k B_k B_k A_k^* B_i B_i=B_i| B_i A_j| A_j^* A_k| B_k| B_k| B_k| A_k^* B_i| B_i. \]
Each pair from $\{B_iA_j,A_j^*B_i,A_i^* A_j\}$ corresponds to the coefficient $a_{ij}$ and each $B_j$ corresponds to $b_j$. It is clear that the coefficient of $w$ in the exansion of $V^m$ is $\alpha(\pi_w,c)$.

\end{proof}

\section{Characterization of quadratic polynomials preserving free infinite divisibility}

\begin{theo} \label{prop:znikaniekumulant}
  The following properties are equivalent for a quadratic form
  $$T (X_1, X_2, \dots, X_N) =\sum_{i,j=1}^N a_{i,j}\X_i\X_j+\sum_{i =1}^N b_iX_i$$
  with selfadjoint system matrix   $A=[a_{i,j}]_{i,j=1}^N\in M_N(\IC)$, $b_i\in\R$. 
\begin{enumerate}
 \item[$(i)$] 
  $T$ preserves free infinite divisibility.
 \item[$(ii)$]
  $A$ is skew symmetric.
\end{enumerate}
\end{theo}

\begin{proof}

 $(ii)\implies (i)$ follows directly from Proposition~\ref{Prop:45}, Theorem~\ref{twr:niskonczonapodzielnoscEL} and the density of compound free Poisson variables (see Section~\ref{prelim:1}).

 We shall show the implication $ (i) \implies (ii) $ through counterexamples.

   First, assume that $ a_{ii} \neq 0$ for some $ i \in [ N]$. Let $ X_i = \beta I+ S$  where $\displaystyle \beta\in \mathbb{R}, \beta \neq - \frac{b_i}{2a_{i, i}} $ and $ S $ is a standard semicircular element (mean 0 and variance 1) and let $ X_j = 0$ for $ j \neq i $. Then 
   \[
  T = a_{i, i} (S+ \gamma I)^2 - \frac{b_i^2}{4 a_{i,i}} 
I,
    \]
   with $ \displaystyle \gamma = \beta + \frac{b_i}{2 a_{i, i}}$.
   But $( S + \gamma I)^2 $ is not freely infinitely divisible when $ \gamma \neq 0$ (see \cite{Eisenbaum:2012}), so the diagonal entries vanish.

    Next, fix $ i, j \in [N], i\neq j $ and let $ X, Y $ be two free self-adjoint free infinitely divisible variables. The shifted variables $ X^\prime = X + t_1I$ and $Y^\prime = Y + t_2I$ are freely infinitely divisible as well, hence so is the quadratic form obtained by evaluating $ T $ for $ X_i = X$, $ X_j = Y $ and $ X_k = 0$ whenever $ k \notin \{i, j\}$, i.e.
    \[
     T = a_{i,j}(X + t_1I) (Y + t_2 I) + a_{j,i}(Y + T_2I) (X + t_1I) + b_i(X + t_1I) + b_j (Y + T_2I).
    \]
   After rearrangement we have
  \begin{align*}
      T&=a_{i,j} X Y+ a_{j,i} Y X +((a_{i,j}+a_{j,i})t_2+b_i)X+((a_{i,j}+a_{j,i})t_1+b_j)Y\\
      &+(a_{i,j}+a_{j,i})t_1t_2+b_it_1+b_jt_2.
  \end{align*}
   
    If $ a_{i, j} \neq - a_{j, i}$, then for 
    $ \displaystyle t_1 = -\frac{b_j}{a_{i, j}+ a_{j, i} }$
    and
    $ \displaystyle t_2 = -\frac{b_i}{a_{i, j}+ a_{j, i} }$
    the 
    reparametrized quadratic form preserves free infinite divisibility if and only if $a_{i,j} X Y+a_{j,i} Y X $ preserves free infinite divisibility, but from \cite{EjsmontLehner:2019:commutators} we know that this implies $a_{i,j}=-a_{j,i}$, which gives a contradiction.
        
\end{proof}

\begin{Rem}\label{computer}
    The condition that the matrix $\big[ a_{i, j}\big]_{i, j=1}^N$ is self-adjoint and skew-symmetric is equivalent to $ T $ being of the form 
    \[
     T(X_1,X_2, \dots, X_N) = \sum_{1 \leq k < l \leq N } \beta_{k, l} \cdot i[X_k, X_l] + \sum_{j=1}^N b_j X_j
    \]
    with $ \beta_{k, l}, b_j \in \mathbb{R}$ for all $j $ and all $ k < l $.

    $ T $ as above is in the Lie algebra generated by $ X_1,X_2, \dots, X_N$. yet, not all self-adjoint elements of this algebra are freely infinitely divisible. A first example was found by Dimitri Shlyakhtenko using AI (GPT5.4), more precisely for free standard semicircular variables $X$ and $ Y $, the self-adjoint element $X + Y - \frac{1}{6} [X,[X, Y]]$ is not freely infinitely divisible, as its 3-rd Hankel determinant is negative. Similar examples were found using FriCAS computer algebra system (code available upon request), such as $ t X + [X, [X, Y]]$ (the 5th Hankel determinant is negative for $t>13$) and $ tY + [X,[X, Y]]$ (the 7th Hankel determinant is negative for $ t< -9$).
    
\end{Rem}

\noindent\textbf{Thanks:} We thank Dimitri Shlyakhtenko for the example in Remark~\ref{computer} and for the inspiring  discussions during the writting of this paper.

\bibliographystyle{alpha}

\end{document}